\documentclass[10pt]{article}

\usepackage{color}
\usepackage{dsfont}
\usepackage{enumerate}
\usepackage{graphicx,float}
\usepackage{caption}
\usepackage{subcaption}
\usepackage{setspace}
\usepackage{hyperref}
\usepackage{color}
\usepackage{diagbox}
\usepackage{amsmath, amsfonts, amssymb, amsthm, amscd,graphicx}

\newtheorem{Theorem}{Theorem}[part]

\newtheorem{Proposition}{Proposition}[part]

\newtheorem{Lemma}{Lemma}[part]

\def \R{\mathbb{R}}

\def \E{\mathbb{E}}
\def \F{\mathbb{F}}

\def \P{\mathbb{P}}

\def \Ac{{\cal A}}

\def \Cc{{\cal C}}

\def \Fc{{\cal F}}

\def \Hc{{\cal H}}

\def \Kc{{\cal K}}

\def \Zc{{\cal Z}}

\def \1{{\mathds 1}}

\newcommand{\essinf}{\mathop{\mathrm{ess\;inf}}}

\def \ni{\noindent}
\def \eps{\varepsilon}
\def \ep{\hbox{ }\hfill$\Box$}

\def\beqs{\begin{eqnarray*}}
\def\enqs{\end{eqnarray*}}
\def\beq{\begin{eqnarray}}
\def\enq{\end{eqnarray}}

 \title{Regulation of renewable resource exploitation}

 \author{Idris Kharroubi\footnote{Sorbonne Universit\'e, LPSM,  \texttt{idris.kharroubi  @upmc.fr}.} \and Thomas Lim\footnote{Laboratoire d'Analyse et Probabilit\'es,
Universit\'e d’Evry-Val d'Essonne and ENSIIE, \texttt{lim@ensiie.fr}.}\and Thibaut Mastrolia\footnote{CMAP, Ecole Polytechnique, IP Paris, \texttt{thibaut.mastrolia@polytechnique.edu}. This author acknowledges the financial supports of the Chaire Financial Risks hosted by the Louis Bachelier Institute and the ANR project PACMAN ANR-16-CE05-0027.}}  
 
             \date{\today}

\begin{document}

 \maketitle

 \begin{abstract}
 
  We investigate the impact of a regulation policy imposed on an agent exploiting a possibly renewable natural resource. We adopt a  principal-agent model in which the Principal looks for a contract, $i.e.$ taxes/compensations, leading the Agent to a certain level of exploitation. For a given contract, we first describe the Agent's optimal harvest using the BSDE theory. Under regularity and boundedness assumptions on the coefficients, we express almost optimal contracts as solutions to HJB equations. We then extend  the result to coefficients with less regularity and logistic dynamics for the natural resource. We end by numerical examples to illustrate the impact of the regulation in our model.   
 \end{abstract}

\vspace{1em}



\section{Introduction}

The exploitation of natural resources is fundamental for the survival and development of the growing human population. However, natural resources are limited since they are either non renewable (e.g. minerals, oil, gas and coal) so that the available quantity is limited, or renewable (e.g. food, water and forests) and in this case the natural resource is limited by its ability to renew itself. In particular, an excessive exploitation of such resources might lead to their extinctions and therefore affect the depending economies with, for instance, high increases of prices and higher uncertainty on the future. Thus, the natural resource manager faces a dilemma: either harvesting intensively the resource to increase her incomes, or taking into account the potential externalities induced by an overexploitation of the resource and impacting her future ability to harvest the resource. It has been nevertheless emphasized in \cite{clark1973profit} that in some cases it is optimal for natural resource manager to harvest until the extinction of the resource. This optimal harvesting strategy thus leads to costs for the global welfare related to the environment degradation.  
\\


Therefore, the management and the monitoring of the exploitation of natural resources are a balance between optimal harvest for the natural resource manager and ecological implications for public organizations. This second issue has attracted a lot of interest, especially from governance institutions. For example in its last annual report on sustainable development, the statistical office of the European Union \textit{Eurostat} dedicates a full section to the question of sustainable consumption and production (see \cite{ESR18}, Section 12). 

The management of natural resources have also attracted a lot of interest from the academic community. 
Many studies on natural resources exploitation tried to describe the possible effect of economic incentives  on the exploitation (see e.g. \cite{EBB93,gjertsen2010incentive,suich2013effectiveness,grafton2006incentive}). These references stress the need of an incentive policy to ensure the sustainability of the resource. However, even if the regulator have access to the abundance level of natural resource, the unobservability of the natural resource manager behavior induces moral hazards. Thus, the regulator's issue is to incentivize the resource manager to optimally reduce the cost of the resource degradation, together with ensuring a minimal incomes for the manager, under moral hazard. To the best of our knowledge, this question has been addressed only in the discrete-time framework (see for instance \cite{gaudet2015management}) without considering any randomness in the dynamics of the resource. The aim of this work is to investigate this problem in continuous time with randomness in the system.

To deal with this issue, we consider a principal-agent model under moral hazard. The first elements of contract theory with moral hazard appeared in the 60's with the articles \cite{arrow1964research,wilson1967structure} in which the mechanisms of controlled management were investigated. Then, it has been extended and named as \textit{agency problem} (see among others \cite{wilson1968theory,ross1973economic}) by considering discrete-time models. Concerning the continuous-time framework, the agency problem with moral hazard has been first studied in \cite{holmstrom1987aggregation} by modelling the uncertainty of risky incomes with a Brownian motion. 

The agency problem can be roughly described as follows. We associate a moral hazard problem with a Stackelberg game in which the leader (named the Principal) proposes at time $0$ a compensation to the follower (named the Agent) given at a maturity $T>0$ fixed by the contract, to manage the wealth of the leader. Moreover, the Principal has to propose a compensation high enough (called the reservation utility) to ensure a certain level of utility for the Agent. Although the Principal cannot directly observe the action of the Agent, the former can anticipate the best reaction effort of the latter with respect to a fixed compensation. Hence the agency problem remains to design an optimal compensation proposed by the Principal to the Agent given all the constraints mentioned above under moral hazard. 

The common approach to solve this problem consists in proceeding in two steps. The first step is to compute the optimal reaction of the Agent given a fixed compensation proposed by the Principal,  \textit{i.e.} solving the utility maximization problem of the Agent. In all the papers mentioned above, the shape of considered contracts is fundamental to solve the Agent problem by assuming that the compensation is composed by
\begin{itemize}
\item a constant part depending on the reservation utility of the Agent,
\item a part indexed by the (risky) incomes of the Principal,
\item the certain equivalent gain of utility appearing in the Agent maximization. 
\end{itemize}
Using the theory of Backward Stochastic Differential Equations (BSDE for short), \cite{cvitanic2018dynamic} proved that this class of \textit{smooth} contracts, having a relevant economic interpretation, is not restrictive to solve the agency problem. The second step consists in solving the Principal problem. Taking into account this optimal reaction of the Agent, the goal is to compute the optimal compensation. As emphasized in \cite{sannikov2008continuous} and then in \cite{cvitanic2016moral,cvitanic2018dynamic}, this problem remains to a (classical) stochastic control problem with the wealth of the Principal and the continuation utility value of the Agent as state variables.\\

In this paper, we identify the natural resource manager as the Agent. The Principal refers to a regulator, which can be a public institution that monitors the resource manager's activities.

The resource manager can either harvest or renew the natural resource. In the first case the production is sold at a given price on the market and in the second case the resource manager pays for each unit of renewed natural resource. To regulate the natural resource exploitation, the Principal imposes a tax/compensation to the Agent depending on the remaining level of resource at the terminal time horizon. We suppose here that the Agent is risk-averse  and we model his preference with an exponential utility function\footnote{See for instance \cite{arrow1971theory} for more details on this kind of utility function and the economical interpretations of it.}. For a given harvesting strategy, the Agent total gain is composed by the cumulated amounts paid/earned by renewing/harvesting the natural resource and the regulation compensation/tax. The Agent's aim is then to maximize the expected utility of his total gain over possible harvesting strategies.
  
On the other side, given the previous optimal harvest of the Agent, the regulator aims at fixing a tax/compensation policy that incentives the Agent to let a reasonable remaining level of natural resource. As a public institution, we assume that the regulator is risk-neutral.\\

The main features to model the dynamic of a renewable natural resource are its birth and death rates and the inter-species competition. Besides, due to random evolution of the population, we consider uncertainty in the available abundance. Following \cite{evans2015protected, alvarez2007optimal,meleard2015some} we choose to model the evolution of the natural resource by a stochastic logistic diffusion.\\

We then focus on the Principal-agent problem. 
We first characterize the Agent behavior for a fixed regulation policy represented by a random variable $\xi$. Following the BSDEs approaches to deal with exponential utility maximization, we get a unique optimal harvesting strategy as a function of the $Z$ component of the solution to a quadratic BSDE with terminal condition $\xi$ (see \cite{rouge2000pricing,hu2005utility}).     

We next turn to the regulator problem which consists in maximizing an expected terminal reward depending on the regulation tax $\xi$ and the level of remaining natural resource according to the Agent's optimal response.  By writing the explicit form of the resource manager's optimal strategy, we turn the regulator problem into a Markov stochastic control problem of a diffusion with controlled drift.  We then look for a regular solution to the related PDE to proceed by verification. However, in our case we face the following three issues.
\begin{itemize}
\item By considering the logistic dynamics for the resource abundance population, the HJB PDE related to the Principal problem involves a term of the form $x^2\partial_x v$ where $x$ stands for the resource population abundance and $v$ is the Principal’s value function. This term, induced by the inter-species competition in the classical logistic case, prevents us from using existence results of regular solutions to PDEs.
\item The shape of the optimal harvest of the manager leads to irregular coefficients for the related PDE, which also prevents from getting regular solutions.
\item Due to the exponential preferences of the Agent, the Principal's admissible strategies need to satisfy an exponential integrability condition. However, the linear preferences of the Principal leads to an optimal contract that is not necessarily exponential integrable. Therefore, the regulator problem might not have an optimal regulation policy. 
\end{itemize}
To deal with these issues, we first study a model for which the inter-species competition coefficient $\mu$ of the population is bounded. Hence, the term $x^2\partial_x v$ is replaced by $x\mu(x)\partial_x v$. We then construct a regular approximation of the Hamiltonian. By considering the related PDE, we derive a regular solution (see Proposition \ref{Approx-Hamiltonian}) together with an almost optimal control satisfying the admissibility condition (see Theorem \ref{thm:regulator}). We notice that our approach can be related to that of Fleming and Soner \cite{FS06}, which consists in an approximating the value function by a sequence of smooth value functions to derive a dynamic programming principle. We next turn to the logistic case $i.e.$ $\mu(x)=x$ for which we show that the almost optimal strategy obtained for a truncation of $\mu$ remains an almost optimal strategy for a large value of the  truncation parameter (see Theorem \ref{thm:logistic}).

We finally illustrate our results by numerical experiments. We compute the almost optimal strategies using approximations of solutions to HJB PDEs and show that the regulation has a significant effect on the level of remaining natural resource. 

\vspace{1mm}

The remainder of the paper is the following. In Section \ref{The natural resource} we describe the considered mathematical problem. We then solve in Section \ref{section:optimaleffortmanager} the manager's problem for a given regulation policy. In Section \ref{section:regulatorpb}, we first provide almost optimal strategies in the case where the coefficient $\mu$ is bounded and we extend our result to the logisitic dynamics. We end Section \ref{section:regulatorpb} by economical insights and numerical experiments.    
\vspace{10mm}

 \subsection*{Notations and spaces}
We give in this part all the notations used in this paper. Let $(\Omega, \Fc, \P)$ be a complete probability space. We assume that this space is equipped with a standard Brownian motion $W$  and we denote by $\F := (\Fc_t)_{t \geq 0}$ its right-continuous and complete natural filtration.

Let $p\geq 1$ and a time horizon $T>0$, we introduce the following spaces
\begin{itemize}
\item  $\mathcal P(\mathbb R)$ (resp. $\mathcal Pr(\mathbb R)$) will denote the $\sigma$-algebra of $\mathbb R$-progressively mesurable, $\mathbb F$-adapted (resp. $\mathbb F$-predictable) integrable processes.
\item $\mathcal S^p_T$ is the set of processes $X,\, \mathcal P(\mathbb R)$-mesurable and continuous satisfying 
$$
\mathbb E[\sup_{0\leq t\leq T} |X_t|^p] < +\infty \;.
$$
\item $\mathbb H^p_T$  is the set of processes $X, \, \mathcal Pr(\mathbb R)$-mesurable satisfying 
$$
\mathbb E\Big[\big(\int_0^T|X_t|^2dt\big)^{\frac p2}\Big] < +\infty \;.
$$
\item For an integer $q\geq 0$, a subset $D$ of $\R^q$ and for any $\nu\in(0,1)$, we denote by $C^{1+\nu}(D)$ the set of  continuously differentiable functions $f:~D\rightarrow\R$ such that
$$
|f|_{1}  =  \sup_{x\in D}\Big(|f(x)| +\sum_{1\leq i\leq q}|\partial_{x_i}f(x)|+ \sup_{x,y\in D}\sum_{1\leq i\leq q}\frac{|\partial_{x_i}f(x)-\partial_{x_i}f(y)|}{|x-y|^\nu}\Big)<\infty\;,
$$
and by $C^{2+\nu}(D)$ the set of twice continuously differentiable functions $f:~D\rightarrow\R$ such that
\beqs
|f|_{2+\nu} & = & \sup_{x\in D}\Big(|f(x)| +\sum_{1\leq i\leq q}|\partial_{x_i}f(x)|+\sum_{1\leq i,j\leq q}|\partial_{x_i,x_j}f(x)|\Big)\\
 & & +\sup_{x,y\in D}\sum_{1\leq i,j\leq q}\frac{|\partial_{x_i,x_j}f(x)-\partial_{x_i,x_j}f(y)|}{|x-y|^\nu}~<~\infty\;.
\enqs 
\end{itemize}

\section{The model}\label{The natural resource}
\subsection{The natural resource}

\noindent We fix a deterministic time horizon $T>0$ and we suppose that the natural resource abundance $X_t^{\mu}$ at time $t\geq 0$ is given  by
\begin{equation}\label{sdeprimal}
X_t^{\mu} = X_0+\int_0^tX_s^{\mu} (\lambda-\mu(X_s^{\mu}))ds+\int_0^t  \sigma X_s^{\mu}dW_s \;,\quad t\in[0,T]\;,
\end{equation}
where $X_0$, $\lambda$ and $\sigma$ are positive constants. The quantities $X_0$  and $\lambda$ correspond to the initial natural resource abundance and the growth rate respectively. The map
$\mu$ represents the competition inside the species considered or more generally an auto-degradation parameter for a natural resource. We assume that the map $\mu$ satisfies the following assumption\\

\noindent $\textbf{(H0)}$ $\mu$ is a map from $\R_+$ to $\R_+$ such that \eqref{sdeprimal} admits a unique strong solution in $\mathcal S^2_T$.\\ 

\noindent Note that Assumption $\textbf{(H0)}$ holds for instance if the map $x\longmapsto x\mu(x)$ is Lipschitz continuous. Another important example
is the so-called logistic equation where $\mu(x)=x$ on $\mathbb R_+$, see for example in \cite{evans2015protected}. In this last case, SDE \eqref{sdeprimal} admits an explicit unique solution that will be denoted in the sequel by $X$ and given by
$$
 X_t = \frac{X_0 e^{(\lambda-\frac{\sigma^2}2)t+\sigma W_t}}{1+X_0\int_0^t e^{(\lambda -\frac{\sigma^2}2)s+\sigma W_s} ds}\;,\quad t\in [0,T]\;.
$$
The ecological interpretation of this model is the following. At time $t$, if the coefficient $\mu(X_t^{\mu})$ is larger than $\lambda$ then the drift of the diffusion is negative. Therefore the abundance of the natural resource $X_t^{\mu}$ decreases in mean. Conversely, if $\mu(X_t^{\mu})$ is smaller than $\lambda$ then the drift of the diffusion is positive. Hence, the abundance $X_t^{\mu}$ increases in mean. For more details see for instance \cite[Proposition 3.4]{meleard2015some}. \\

\noindent More general models can be used in practice and one of the main challenges, see \cite{meleard2015some}, is to rely branching processes with  birth and death intensities to the solutions of continuous SDEs.
\subsection{The Agent's problem }\label{section:modelagent}
We consider an agent who tries to make profit from the natural resource. We suppose that this agent owns facilities to either harvest or renew this resource. We assume that his action happends continuously in time and we denote by $\alpha_t$ his intervention rate at time $t$, \textit{i.e.} the abundance $X_t^\mu$ will decrease of an amount $\alpha_tX_t^\mu$ per unit of time. This means that if the intervention rate $\alpha_t$ is positive (resp. negative), the Agent harvests (resp. renews) the natural resource.  We denote by $\mathcal A$ the set of $\F$-adapted processes defined on $[0,T]$ and valued in $[-\underline M, \overline M]$ where $\underline M$ and $\overline M$ are two nonnegative constants. If the Agent is prohibited to renew the resource then $\underline M = 0$.
This set $\Ac$ is
 called the set of admissible actions.\\

 To take into account the control $\alpha$ of the Agent on the natural resource abundance, we introduce the probability measure $\P^\alpha$ defined by its density $H^\alpha$ w.r.t.  $\P$ given by
$$
{d\P^\alpha \over d\P}\Big|_{\Fc_T}  :=  H^\alpha_T \;,
$$
where the process $H^\alpha$ is defined by
$$
H^\alpha_t := \exp\Big( -\int_0^t{\alpha_s\over \sigma } dW_s-{1\over 2}\int_0^s \Big|{\alpha_s\over \sigma}\Big|^2 ds \Big) \;,\quad t\in[0,T]\;.
$$
In the sequel, we denote by $\E^{\alpha}$ and  $\E^{\alpha}_t$ the expectation and conditional expectation given $\Fc_t$ respectively, for any $t\in[0,T]$, under the probability measure $\P^\alpha$.

For $\alpha\in\Ac$, we get from Girsanov Theorem (see e.g. Theorem 5.1 in \cite{KS98}) that  
the process $W^\alpha$ defined by
$$
W_t^\alpha  :=  W_t+ \int_0^t\frac{\alpha_s}{\sigma}ds\;,\quad t\in[0,T]\;, 
$$
is a Brownian motion under the probability $\P^\alpha$. Thus, for a given admissible effort $\alpha\in \Ac$,  the  dynamics of $X$ can be rewritten under the probability $\P^\alpha$ as
\begin{equation*}
X_t^{\mu} = x+\int_0^t\big(X_s^{\mu} (\lambda-\mu(X_s^{\mu})) - \alpha_sX_s^{\mu}\big)ds+\int_0^t\sigma X_s^{\mu}dW^\alpha_s \;,\quad t\in[0,T]\;.
\end{equation*}
This new dynamics reflects the evolution of the population with a rate $\alpha_t$ per unit of time. Hence, $ \alpha_tX_t^{\mu}$ has to be seen as the speed of the exploitation of the natural resource at time $t$.

\vspace{2mm}

We then are given a price function $p:~\R_+\rightarrow\R_+$ and we suppose that the price per unit of the natural resource on the market is given by $p(X_t^{\mu})$ at time $t\geq 0$. We make the following assumption on the price function $p$.

\vspace{2mm}

\ni \textbf{(H$_p$)} There exists a constant $P$ such that 
$p(x)x  \leq  P$ for all $x\in[0,+\infty)$.

\vspace{2mm}

This price function $p$ allows to take into account the dependence w.r.t. the abundance (the more abundant the resource is, the cheaper it will be and conversely).  
 Such a price dependence has already been used to model liquidity effects on financial market, where empirical studies showed that the impact is of the form $p(x)=Pe^{-\beta_1 x^{\beta_2}}$, $x\in\R_+$, for some positive constants $P$, $\beta_1$  and $\beta_2$ (see e.g. \cite{ATHH05,LFM03}). In particular, \textbf{(H$_p$)} is satisfied for this type of dependence. 
 Another basic example for which \textbf{(H$_p$)} holds is the case $p(x)=Px^{-1}$, $x>0$. This last example reflects the inability to buy the natural resource once it is extinct.  \\

We assume that the  manager sells the harvested resource on the market at price $p(X_t^{\mu})$ per unit at time $t$ if $\alpha_t$ is positive, and pays the price $p(X_t^{\mu})$ per unit of natural resource at time $t$ if $\alpha_t$ is negative to renew this one. This provides the global amount $\int_0^Tp(X^\mu_t)X^\mu_t\alpha_tdt$ over the time horizon $[0,T]$.\\

We also suppose that giving an effort is costly for the manager and we consider the classical quadratic cost function $k:\R\rightarrow \R_+$ given by $k(\alpha)=\frac{|\alpha|^2}2,\; \alpha\in \mathbb R$. Thus, the Agent is penalized by the instantaneous amount $k(\alpha_t)$ per unit of time for a given effort $\alpha\in \Ac$. This leads to the global payment $\int_0^Tk(\alpha_t)dt$ over the considered time horizon $[0,T]$. \\

\noindent In our investigation, we recall that the activity of the natural resource manager is regulated by an institution (usually an environment administration) who is taking care about the size of the remaining natural resource. To avoid an over-exploitation, the regulator imposes a tax on the Agent depending on the remaining resource. This tax amount is represented by an $\Fc_T$-measurable random variable $\xi$ and is paid at time $T$. Note that $\xi$ can be either positive or negative. In this last case, it means that the regulator gives a compensation to the manager. \\

\noindent Throughout the paper we assume that the Agent's preferences are given by the exponential utility function $u_A$ defined by
$$
 u_A (x) := -\exp\big(-\gamma x\big)\;,\quad x\in \R\;,
$$
where $\gamma$ is a positive constant corresponding to the risk aversion of the Agent. We define the value function $V_A(\xi)$ of the Agent associated to the taxation policy $\xi$ by
 \begin{equation}\label{pb:agent}
V_A(\xi)  :=  \sup_{\alpha \in \Ac} \E^{\alpha} \Big[ - \exp \Big( -\gamma \big( \int_0^T p(X_s^{\mu})X^\mu_s \alpha_s ds - \int_0^T \frac{|\alpha_s|^2}2 ds-\xi \big) \Big) \Big]\;.
\end{equation}

\noindent For a fixed tax $\xi$, we denote by $\mathcal A^*(\xi)$ the set of efforts $\alpha^*\in \Ac$ satisfying the following equality 
$$
 \E^{{\alpha^*}} \Big[ - \exp \Big( -\gamma \big( \int_0^T p(X_s^{\mu})X_s^{\mu}\alpha^*_s ds - \int_0^T \frac{|\alpha^*_s|^2}2 ds  - \xi\big) \Big) \Big]  = 
V_A(\xi) \;.
$$
An effort $\alpha^* \in \Ac^*(\xi)$ is said to be optimal for the fixed tax $\xi$.

\subsection{The Principal's problem}

The aim of the regulator is to stabilize the resource population at a fixed target size at the maturity $T$. For that, a tax $\xi$ is chosen to incentivize the Agent to manage the natural resource so that the remaining population is close to the targeted size. Hence, the regulator benefits from the tax paid by the Agent and is penalized through a cost function $f$ depending on the size of the resource at maturity $T$. The expected reward  under the action $\alpha\in \Ac$ of the Agent is then  given by
$$
 \E^{{\alpha}} \big[ \xi - f(X_T^{\mu})\big] \;.
$$
Typically, we have in mind $f(x)=c(\beta-x)^+$ meaning that the regulator targets a population size $\beta>0$ at time $T$ for the sustainability of the resource and pays the cost $c$ per unit if the natural resource is over-consumed. This function $f$ can be seen as the amount that the regulator must pay to reintroduce the missing resource.
 
\noindent  We suppose that the resource manager is rational. Therefore, the Principal anticipates that for a tax $\xi$, the Agent will choose an effort $\alpha$ in the set $\mathcal A^*(\xi)$. Note that this set is not necessarily reduced to a singleton\footnote{ In our investigation, we will show that the set $\mathcal A^*(\xi)$ is reduced to a single element.}, hence, as usual in moral hazard problems (see for instance \cite{holmstrom1987aggregation} for the formulation of the moral hazard problem), the regulator solves
\begin{equation}\label{pb:regulator} 
\sup_{\xi} V^P(\xi) \;, \text{ with }V^P(\xi) = \sup_{\alpha\in \mathcal A^*(\xi)} \E^{{\alpha}} \Big[\xi-f(X_T^{\mu}) \Big],
\end{equation}
where $\xi$ lives in a set of suitable contracts defined in the following section.
 
\subsection{Class of contracts and utility reservation}

\noindent We now introduce a reserve utility $R$ which is a negative constant. This reserve means that the regulator cannot penalize too strongly the Agent for economical reasons so that the utility $V_A(\xi)$ expected by the Agent has to be greater than $R$. For instance, we can choose $R$ such that the regulator monitors the Agent by promising the same expected utility as the case without regulation (see Section \ref{application:reservation} for more details). This example reflects a non-punitive taxation policy in which the regulator purely monitors the activities of the Agent. In our model, the sign of the tax $\xi$ is on purpose. This means that the natural resource manager pays the fee to the regulator when $\xi$ is positive and conversely, the regulator compensates the Agent's activity when $\xi$ is negative. Moreover, we need to impose an exponential integrability on the tax $\xi$to ensure the well-posdness of $V_A(\xi)$. We therefore introduce the class $\Cc^\mu_R$ of admissible taxes  defined as the set of $\mathcal F_T$-measurable random variables $\xi$ such that
\begin{equation}\label{reserveutility}
V_A(\xi)  \geq  R \;,
\end{equation}
and there exists a constant $\gamma'>  2\gamma$ such that
\begin{equation}\label{integrabilite:zeta}
 \mathbb E\big[\exp(\gamma' |\xi|)\big]  <  +\infty \;.
\end{equation}
This last condition 
 is very convenient since it allows to deal with the problem by using the theory of BSDEs. Moreover, a straightforward application of Cauchy-Schwarz inequality ensures that the optimization problems $V_A(\xi)$ and $V^P(\xi)$ take finite values.

\section{Optimal effort of the natural resource's manager}\label{section:optimaleffortmanager}

\noindent We first solve the optimal problem of the Agent \eqref{pb:agent} under taxation policy $\xi\in \Cc_R^\mu$.
As in \cite{cvitanic2018dynamic}, the following result shows that solving the Agent problem gives both an optimal effort $\alpha^*$ and a particular representation of the tax $\xi$ with respect to the solution of a BSDE.

\begin{Theorem}\label{thm:agent}
Let $\xi \in \mathcal C^\mu_R$ and Assumption \textbf{(H$_p$)} be satisfied. There exists a unique pair $(Y_0,Z)\in (-\infty,\tilde R]\times \mathbb H^2_T$ with $\tilde R:=\frac{\log(-R)}{\gamma}$ such that 

\begin{enumerate}[(i)]
\item the tax has the following decomposition 
\begin{equation} \label{decomposition:tax}
\xi = Y_0-\int_0^T\big(g(X_t^{\mu},Z_t)+\frac{ \sigma^2}{2}\gamma |Z_t|^2 \big) dt+\int_0^T \sigma Z_t dW_t \;,
\end{equation}
where $g$ is defined for any $(x,z)\in \mathbb R_+\times \mathbb R$ by
$$
g(x,z) = \frac{|  a^*(x,z)|^2}2-p(x) x a^*(x,z)-  a^*(x,z) z\;,
$$
and
\begin{equation} \label{alphastar:thm}
  a^*(x,z) = \big((p(x)x+z)\vee (-\underline M)\big)\wedge \overline M \;,
\end{equation}
\item the value of the Agent is given by
$$
V_A(\xi) =  - \exp(\gamma Y_0) \;, 
$$
\item the process $\alpha^*(\xi)$ defined by $\alpha^*_t(\xi)=  a^*(X_t^{\mu},Z_t)$ is the unique optimal effort associated with the tax $\xi$ given by \eqref{decomposition:tax}.
\end{enumerate}
\end{Theorem}
\proof The proof is divided in three steps and is related to the BSDE associated with the value function of the Agent. We first introduce a dynamic extension of the optimization problem \eqref{pb:agent}. We denote by $J(t,\xi)$ the dynamic value function of the Agent at time $t$ for a tax $\xi$ which is defined by
\begin{equation*}
J(t, \xi) := \essinf_{\alpha \in \Ac} \E^{\alpha}_t \Big[ \exp \Big( -\gamma \big( \int_t^T p(X_s^{\mu})  X_s^{\mu}  \alpha_s ds - \int_t^T k(\alpha_s) ds - \xi \big) \Big) \Big] \;.
\end{equation*}
Note that $V_A(\xi) = - J(0,\xi)$.
 \vspace{0.2em}

\noindent \textit{Step 1. Dynamic utility of the Agent and BSDE.} We characterize $J(\cdot,\xi)$ as the unique solution of a BSDE and we derive the optimal control by using comparison results.\\

\noindent Let $\alpha \in \Ac$, we introduce the process $J^\alpha(\xi)$ defined by
$$
J^\alpha_t(\xi) := \E_t^{\P^\alpha} \Big[ \exp \Big( -\gamma \big(  \int_t^T p(X_s^{\mu})X_s^{\mu}\alpha_sds -\int_t^T k(\alpha_s ) ds -\xi \big) \Big) \Big]\;,
$$
so that 
\begin{equation}\label{jjalpha}
J(t, \xi) := \essinf_{\alpha \in \Ac_t} J^\alpha_t(\xi) \;.
\end{equation}

\noindent \textit{Step 1a. Martingale representation and integrability.}

\noindent We know that the process $H_t^\alpha(\exp(\gamma \int_0^t \big(k(\alpha_s)-p(X_s^{\mu})X_s^{\mu}\alpha_s\big) ds) J^\alpha_t(\xi))_{0 \leq t \leq T}$ is a $(\P, \F)$-martingale. In view of the condition \eqref{integrabilite:zeta}, there exists $\varepsilon>0$ and $q>1$ such that $(2+\varepsilon)q\gamma \leq \gamma'$. Hence, for $p>1$ such that $\frac1p+\frac1q=1$, since $\alpha$ is bounded and Condition \eqref{integrabilite:zeta} is satisfied, we get from H\"older's inequality
\beqs
\mathbb E[|\overline{H}_T^\alpha J_T^\alpha(\xi)|^{2+\varepsilon}]&\leq & \mathbb E[|\overline{H}_T^\alpha|^{(2+\varepsilon)p}]^{\frac1p}  \mathbb E[|e^{(2+\varepsilon)q\gamma\xi}|]^{\frac1q} ~< ~ +\infty \;,
\enqs
where $\overline H_t^\alpha:=H_t^\alpha \exp \Big(\gamma \int_0^t \big(k(\alpha_s )-p(X_s^{\mu})X_s^{\mu}\alpha_s \big)ds \Big)$. Hence, by using Doob's maximal inequality, $\overline H^\alpha J^\alpha(\xi)\in \mathcal S^{2+\varepsilon}$. So by using the martingale representation theorem, we know there exists a process $\overline Z^\alpha \in \mathbb H^{2+\varepsilon}_T$ such that 
\beqs
{\overline H}_t^{\alpha}J^\alpha_t(\xi) &=& J^\alpha_0 + \int_0^t  \sigma \overline Z^\alpha_s  d W_s \;,  \quad t\in [0,T] \;.
\enqs
\noindent Therefore, $J^\alpha$ satisfies
\beqs
dJ^\alpha_t(\xi) &=& (\alpha_t \tilde Z^\alpha_t - \gamma (k(\alpha_t) -p(X_t^{\mu})X_t^{\mu}\alpha_t)J^\alpha_t (\xi))dt + \sigma  \tilde Z^\alpha_t  dW_t \;,
\enqs
where $\tilde Z^\alpha_t = \frac{\overline Z^\alpha_t}{{\overline H}_t^{\alpha}} +J_t^\alpha \frac{\alpha_t}{\sigma^2}$ for any $t \in [0,T]$, and $J_T^\alpha (\xi)= \exp(\gamma\xi)$.\\

 We now prove that  $\tilde Z^\alpha\in \mathbb H_T^2$. From \eqref{integrabilite:zeta}, the boundedness of $\alpha$ and Assumption \textbf{(H$_p$)}, there exists a positive constant $C>0$ such that
\beqs
\mathbb E \Big[\int_0^T  | \frac{\overline Z^\alpha_t}{{\overline H}_t^{\alpha}} +J_t^\alpha \frac{\alpha_t}{\sigma^2}|^2 dt \Big]&\leq & 2\Big( \mathbb E\Big[\int_0^T  | \frac{\overline Z^\alpha_t}{{\overline H}_t^{\alpha}}|^2dt\Big] +  \mathbb E\Big[\int_0^T |J_t^\alpha \frac{\alpha_t}{\sigma^2}|^2 dt\Big]\Big)\\
&\leq & C\Big(1+ \mathbb E \Big[\int_0^T  | \frac{\overline Z^\alpha_t}{{\overline H}_t^{\alpha}}|^2dt\Big] \Big) \;.
\enqs
We set $\tilde q:=1+\frac{\varepsilon}2$ and $\tilde p>1$ such that $\frac1{\tilde p}+\frac1{\tilde q}=1$. Using H\"older and BDG Inequalities and since $\overline Z^\alpha \in \mathbb H^{2+\varepsilon}_T$, we get
\beqs
\mathbb E [\int_0^T  | \frac{\overline Z^\alpha_t}{{\overline H}_t^{\alpha}}|^2dt] &\leq &\mathbb E[ \sup_{t\in [0,T]} |({\overline H}_t^{\alpha})^{-1}|^2\int_0^T  | \overline Z^\alpha_t|^2dt]\\
&\leq & \mathbb E[ \sup_{t\in [0,T]} |({\overline H}_t^{\alpha})^{-1}|^{2\tilde p}]^{\frac1{\tilde p}} \mathbb E[\big(\int_0^T  | \overline Z^\alpha_t|^2dt\big)^{\tilde q}]^{\frac{1}{\tilde q}}\\
&<&+\infty \;.
\enqs
Consequently, we get $\tilde Z^\alpha\in \mathbb H_T^2$.\vspace{0.5em}

\noindent \textit{Step 1b. Comparison of BSDEs and optimal effort.} We now turn to the characterization of the solution to \eqref{jjalpha} by a BSDE. We introduce the following BSDE
\beq\label{bsde:jinf}
d\underline{J}_t(\xi) &=&-\inf_{a\in [- \underline M, \overline M]} G(X_t^{\mu},\underline{J}_t (\xi),\underline{\tilde Z}_t, a)dt + \sigma  \underline{\tilde Z}_t  dW_t \;,\quad \underline{J}_T (\xi)= \exp(\gamma\xi) \;,
\enq
where 
\beqs
G(x,j,\tilde z,a) &:=&  \gamma (k(a) -p(x)xa)j- a\tilde z \;.
\enqs
 This BSDE has a Lipschitz generator and square integrable terminal condition from \eqref{integrabilite:zeta}. Therefore it admits a unique solution in $\mathcal S^2\times \mathbb H_T^2$. Moreover, for any $\alpha\in \mathcal A$, we notice that $(J^\alpha(\xi),\tilde Z^\alpha)$ satisfies the following BSDE
\beqs
d{J}^\alpha_t(\xi) &=&-G(X_t^{\mu},{J}^\alpha_t (\xi),{\tilde Z}^\alpha_t, \alpha_t)dt + \sigma  {\tilde Z}^\alpha_t  dW_t \;,\quad {J}^\alpha_T (\xi)= \exp(\gamma\xi) \;.
\enqs
By classical comparison Theorem, we have 
\beqs
 \underline{J}_t(\xi) &\leq & J(t,\xi) \;,\quad \forall t\in [0,T] \;. 
  \enqs
Then, we notice that BSDE \eqref{bsde:jinf} can be rewritten 
\beqs
 d\underline{J}_t(\xi) &= & -G\Big(X_t^{\mu},\underline{J}_t (\xi),\underline{\tilde Z}_t, {a^*}\big(X_t^{\mu},\frac{\underline{\tilde Z}_t}{\gamma\underline{J}_t (\xi)}\big )\Big)dt + \sigma  \underline{\tilde Z}_t  dW_t \;,\quad \underline{J}_T (\xi) ~= ~\exp(\gamma\xi)\;.
\enqs

\noindent In particular, we have $\underline{J}(\xi)= J^{ a^*\big(X^{\mu},\frac{\underline{\tilde Z}}{\gamma\underline{J} (\xi)}\big)}(\xi)$ by uniqueness of the solution to BSDE (\ref{bsde:jinf}). Therefore, we get
\beq\label{optim:Jpreuve}
\underline{J}_t(\xi) &= & J(t,\xi) \quad  \text{ and } \quad a^*\big(X^{\mu},\frac{\underline{\tilde Z}}{\gamma\underline{J} (\xi)}\big)\in \mathcal A^*(\xi)\;.
\enq\vspace{0.2em}
\noindent We now prove that this optimal effort is unique. Let $\tilde \alpha\in \mathcal A$ be an other optimal effort, then we have 
\beqs
J_0^{\tilde \alpha} &=& J_0^{a^*\big(X^{\mu},\frac{\underline{\tilde Z}}{\gamma\underline{J} (\xi)}\big)}\;.
\enqs
 From strict comparison Theorem (see for instance \cite[Theorem 2.2]{elkaroui1997backward}) we get $J^{\tilde \alpha}= J^{a^*\big(X^{\mu},\frac{\underline{\tilde Z}}{\gamma\underline{J} (\xi)}\big)}$ and 
\beqs
G\Big(X^{\mu},\underline{J} (\xi),\underline{\tilde Z}, {a^*}\big(X^{\mu},\frac{\underline{\tilde Z}}{\gamma\underline{J} (\xi)}\big )\Big) &=& G\Big(X^{\mu},\underline{J} (\xi) ,\underline{\tilde Z}, \tilde \alpha\Big) \;,\quad  dt\otimes d\mathbb P-a.e.
\enqs
By the uniqueness of the minimizer of $G(X_t^{\mu},\underline{J}_t (\xi),\underline{\tilde Z}_t,\cdot),$ we deduce that ${\tilde \alpha}= {a^*\big(X^{\mu},\frac{\underline{\tilde Z}}{\gamma\underline{J} (\xi)}\big)}.$\vspace{0.5em}

\noindent \textit{Step 2. Representation of $\xi$.}

\noindent Since, by definition, the process $J^{a^*\big(X^{\mu},\frac{\underline{\tilde Z}}{\gamma\underline{J} (\xi)}\big)}$ is positive, we can define the processes $Y$ and $Z$ by
\beq\label{def:YJ}
Y~:=~ \frac{\log\Big(J^{a^*\big(X^{\mu},\frac{\underline{\tilde Z}}{\gamma\underline{J} (\xi)}\big)}\Big)}{\gamma} & \text{ and } &  Z~:=~\frac{\underline{\tilde Z}}{\gamma J^{a^*\big(X^{\mu},\frac{\underline{\tilde Z}}{\gamma\underline{J} (\xi)}\big)}} \;.\enq We obtain
\begin{align*}
 dY_t &= -\big(k(  a^*(X^{\mu},Z_t ))-p(X_t^{\mu})X_t^{\mu}  a^*(X^{\mu},Z_t )-  a^*(X^{\mu},Z_t ) Z_t+\frac{ \sigma^2}{2}\gamma |Z_t|^2 \big)dt + \sigma  Z_t dW_t \;,\\
 Y_T&=\xi \;.
\end{align*}

\noindent We first prove $Y\in \mathcal S^2$. Note for any $t\in [0,T]$, by using Jensen inequality, we have 
\beqs
 \frac1\gamma\log(\underline J_t(\xi)) &\geq &\mathbb E^{\alpha^*}_t \Big[   \int_t^T  \big( k(\alpha^*_s )  - p(X_s^{\mu})X_s^{\mu}\alpha^*_s \big)ds + \xi   \Big] ~\geq ~ \mathbb E_t^{\alpha^*}[ \xi]  -  TPM \;,
 \enqs
where $\alpha^*$ stands for $a^*\big(X^{\mu},\frac{\underline{\tilde Z}}{\gamma\underline{J} (\xi)}\big)$ and $M = \underline M \vee \overline M$. We then notice 
\begin{itemize}
\item if $\underline J_t(\xi)\geq 1$ we have
\beqs\label{estim1:sec3}
0~~\leq~~\log(\underline J_t(\xi)) & \leq & \underline J_t(\xi) \;,
\enqs 
\item  if $0\leq \underline J_t(\xi)< 1$ we have
\beqs\label{estim2:sec3}
\Big|\frac1\gamma\log(\underline J_t(\xi))\Big| & \leq & TPM+\mathbb E_t^{\alpha^*}[|\xi|]  \;.
\enqs
\end{itemize}
 Hence, there exists a constant $C>0$ such that
\beqs
 \Big| \frac1\gamma\log(\underline J_t(\xi))\Big|^2 & \leq & C(1+\mathbb E_t^{\alpha^*}[ |\xi|]^2)+\frac{1}{\gamma^2}|\underline J_t(\xi)|^2 \;, \quad t \in [0,T]\;.
\enqs
From Young inequality, we get
\beqs
\mathbb E\Big[\sup_{t\in [0,T]} \Big| \frac1\gamma\log(\underline J_t(\xi))\Big|^2\Big] & \leq & 2C\Big(1+\E\Big[\sup_{t\in [0,T]}\Big( {H^{\alpha^*}_T\over H^{\alpha^*}_t}\Big)^{4}\Big]+\E[|\xi|^{4}]\Big)+\frac{1}{\gamma^2} \mathbb E\Big[\sup_{t\in [0,T]}|\underline J_t(\xi)|^2\Big] \;.
\enqs
Since $\alpha^*$ is bounded, we have $\E\Big[\sup_{t\in [0,T]}\Big( {H^{\alpha^*}_T\over H^{\alpha^*}_t}\Big)^{4}\Big]<+\infty$. 

Using $\underline J(\xi)\in \mathcal S^2$,  we obtain
\[
\mathbb E[\sup_{t\in [0,T]} \Big| \frac1\gamma\log(\underline J_t(\xi))\Big|^2 ] <  +\infty \;.
\]
Which implies $Y\in \mathcal S^2$.\\

 We now check $Z\in \mathbb H_T^2$. To this end, we use a localization procedure by  introducing the sequence of  stopping times $(\tau_n)_{n\geq1}$ defined by 
\beqs
 \tau_n & := & \inf\Big\{t\in [0,T],\; \int_0^t |Z_s|^2ds\geq n\Big\}\wedge T \;,
\enqs 
for any $n\geq 1$. Similarly to the proof of \cite[Theorem 2]{briand2006bsde}, we apply It\^o's Formula to $\iota(|Y|)$ where $\iota(x)=\frac1{\gamma^2}(e^{\gamma x}-\gamma x-1)$ for $x\in\R$. We obtain 
\beqs
\iota(|Y_0|)&=& \iota(|Y_{\tau_n}|) +\int_0^{\tau_n} \Big(\iota'(|Y_s|)\textrm{sgn}(Y_s)\big(g(X_s^\mu,Z_s)+{\sigma^2\gamma \over 2} |Z_s|^2\big)-\frac{1}2\iota''(|Y_s|)\sigma^2 |Z_s|^2\Big)ds\\
 & & -\int_0^{\tau_n} \sigma \iota'(|Y_s|)\textrm{sgn}(Y_s)Z_s dW_s \;.
\enqs
Since $\iota''-\gamma \iota'=1$ and $\iota'(x)\geq 0$ for $x\geq 0$, we get from BDG and Young inequalities
\beq\label{las-estim-Agentpb}
\E\Big[\int_0^{\tau_n}  |Z_s|^2 ds\Big] & \leq & C\Big(1+\E\Big[\sup_{t\in[0,T]}e^{\gamma|Y_t|}+\int_0^Te^{\gamma|Y_t|}\big(1+|Y_s|\big)\Big]\Big)\;.
\enq
From the definition of $Y$ and since $\alpha^*$ is bounded, there exists a constant $C$ such that
\beqs
2\gamma|Y_t| & \leq & C+ 2\gamma\E_t^{\alpha^*}[|\xi|]\;.
\enqs
Using Jensen and H\"older inequalities we get another constant $C'$ such that
\[
\E\Big[\sup_{t\in[0,T]}e^{2\gamma|Y_t| }\Big]  \leq  C'\E\Big[\sup_{t\in[0,T]}\Big({H_T^{\alpha^*}\over H_t^{\alpha^*}}\Big)^{\gamma'\over \gamma'-\gamma}\Big]^{ \gamma'-\gamma\over \gamma'} \E\Big[e^{2\gamma'|\xi| }\Big]^{\gamma \over \gamma'}\;.
\]
Since $\alpha^*$ is bounded, we have $\E\Big[\sup_{t\in [0,T]}\Big( {H^{\alpha^*}_T\over H^{\alpha^*}_t}\Big)^{\gamma'\over \gamma'-\gamma}\Big]<+\infty$ and we get 
from \eqref{integrabilite:zeta}
\beqs
\E\Big[\sup_{t\in[0,T]}e^{2\gamma|Y_t|}\Big] & < & +\infty\;.
\enqs
Sending $n$ to $\infty$ in \eqref{las-estim-Agentpb}, we get from Fatou's Lemma $Z\in\mathbb H ^2$.

\vspace{0.5em}
\noindent \textit{Step 3. Conclusion.} We directly deduce $(ii)$ and $(iii)$ from \eqref{def:YJ} together with \eqref{optim:Jpreuve} given that $(i)$ has been proved in Step 2.

\ep

\section{The problem of the regulator}\label{section:regulatorpb}
In this section, we focus on the regulation policy. In view of \eqref{pb:regulator} and Theorem \ref{thm:agent} the regulator's problem turns to be

\beq \label{pb:regul}
V^P_R &=&\sup_{\xi \in \Cc^\mu_R} \mathbb E^{\alpha^*(\xi)}[\xi-f(X_T^{\mu})] \;.
\enq
We first provide almost optimal contracts for a bounded parameter $\mu$ by a PDE approach. We then extend the study to the logistic  case with $\mu(x)=x$.

\subsection{Almost optimal strategies for bounded auto-degradation and cost parameters}
We introduce the following class of contracts
\beq
\nonumber \Xi^\mu& := &\Big\{Y_T^{Y_0,Z,\mu}=Y_0-\int_0^T\big(h(X_t^{\mu},Z_t)+\frac{ \sigma^2}{2}\gamma |Z_t|^2\big)dt+\int_0^T\sigma Z_t dW_t\;,\\
&& \qquad Y_0\leq \tilde R\;, ~Z\in \mathcal Z \Big\} \;,\label{defXimu} \enq
where $\mathcal Z$ denotes the subset of predictable processes of $\mathbb H^2_T$ such that 
\beq\label{integr:z}
\mathbb E[\exp(\gamma' |Y_T^{Y_0,Z,\mu}|)] &< &+\infty\;,
\enq
for some $\gamma'>  2\gamma$ and we recall that $\tilde R={\log(-R)\over \gamma}$. When $\mu$ is the identity, we omit the exponent $\mu$ in the previous definitions.\\

\noindent From Theorem \ref{thm:agent}, constraint \eqref{reserveutility} and integrability conditions \eqref{integrabilite:zeta} and \eqref{integr:z}, the set $\mathcal C^\mu_R$ coincides with $\Xi^\mu$ so that the regulator's problem  \eqref{pb:regul} becomes
\beq \label{value:reg:ximu}
V^P_R&=&\sup_{Y_0\leq \tilde R, ~Z\in \mathcal Z} \mathbb E^{{a^*(X^\mu,Z)}}[Y^{Y_0,Z,\mu}_T-f(X_T^{\mu})]\;,
\enq
with  
\beqs
Y_t^{Y_0,Z,\mu} &=& Y_0-\int_0^t \big(k(\alpha^*_s)-p(X^{\mu}_s) X_s^{\mu} \alpha_s^*+\frac{ \sigma^2}{2}\gamma |Z_s|^2 \big)dt+\int_0^t \sigma Z_s dW^*_s \;,\quad t\in[0,T]\;,
\enqs
where $W^*$ stands for $W^{ a^*(X^{\mu},Z)}$.
We notice that the function to maximize in $V^P_R$ is nondeacreasing w.r.t. the variable $Y_0$. Therefore the constraint $Y_0\leq \tilde R$ is saturated and \eqref{value:reg:ximu} can be rewritten under the following form 
\beq\label{value:reg:ximu2}
V_R^P &=& \sup_{ Z\in \mathcal Z} \mathbb E^{a^*(X^\mu,Z)}[Y^{\tilde R,Z,\mu}_T-f(X_T^{\mu})] \;.
\enq
To construct a solution to the problem \eqref{value:reg:ximu2}, we introduce the related HJB PDE given by

\begin{equation}\label{hjb}\begin{cases}
&-\partial_t v-H\Big(x,\partial_{x} v(t,x), \partial_{xx} v(t,x)\Big) =0 \;, \quad (t,x)\in [0,T)\times \mathbb R_+^* \;,\\
&v(T,x)=- f(x) \;, \quad x\in \mathbb R_+^* \;,
\end{cases}\end{equation}
where the Hamiltonian $H$ is given by
\beqs
H(x,\delta_1,\delta_2) & = & \sup_{z\in\R}\left\{xp(x)  a^*(x,z) -k(  a^*(x,z))-\frac{ \sigma^2}{2}\gamma z^2 + x(\lambda-\mu(x)-  a^*(x,z))\delta_1 \right\}\\
 & & 
+\frac{\sigma^2}{2} x^2 \delta_2\;,\quad(x,\delta_1,\delta_2)\in \R_+^*\times\R\times\R\;,
\enqs
and $a^*$ is given by \eqref{alphastar:thm}.
We first extend PDE \eqref{hjb} to the whole domain $[0,T]\times\R$ by considering the change of variable $w(t,y):= v(t,e^y)$ for any $(t,y)\in [0,T]\times \mathbb R$. We get the following PDE
\begin{equation}\label{pde:w}\begin{cases}
&-\partial_t w-\mathcal H\Big(y,\partial_{y} w(t,y), \partial_{yy} w(t,y)\Big)= 0\; ,\quad (t,y)\in [0,T)\times \mathbb R \;,\\
&w(T,y)= -f(e^y)\;,\quad y\in \mathbb R \;,
\end{cases}\end{equation}
where
\beqs
\mathcal H(y,\delta_1,\delta_2) & := & \sup_{z\in\R}\left\{ e^yp(e^y)a^*(e^y,z)-{a^*(e^y,z)^2\over 2}-\frac{ \sigma^2}{2}\gamma z^2 + (\lambda-{\sigma^2\over 2}-\mu(e^y)-a^*(e^y,z))\delta_1 \right\}\\
 & & 
+\frac{\sigma^2}{2}  \delta_2\;,\quad(y,\delta_1,\delta_2)\in \R\times\R\times\R\;.
\enqs
Our aim is to construct a regular solution to this PDE to proceed by verification.
Unfortunately, the coefficients of PDE \eqref{pde:w} are not smooth enough to do so. 
To overcome this issue, we provide a smooth approximation $\mathcal H_\eps$ of $\mathcal H$ for which we get regular solutions. 

\vspace{2mm}

Moreover, we introduce the following assumption, which ensure that the optimal control derived from the PDE satisfies the admissibility condition, \textit{i.e.} belongs to $\mathcal Z$.

\vspace{2mm}

\noindent \textbf{(H')} There exists $\nu\in(0,1)$ such that  
\begin{itemize}
\item[$(i)$] the map $y\mapsto \mu(e^y)$ belongs to $C^{1+\nu}(\R)$,
\item[$(ii)$]the map $y \mapsto f(e^y)$ belongs to $C^{2+\nu}(\R)$,

\item[$(iii)$]the map $y\mapsto p(e^y)e^y$ belongs to $C^{1+\nu}(\R)$.
\end{itemize}

\begin{Proposition}\label{Approx-Hamiltonian}
Under \textbf{(H')}, there exists a family $\{\mathcal H_\eps,~\eps>0\}$ of functions from $\R^3$ to $\R$ such that
 the PDE
\begin{equation}\label{pde:w-eps}\begin{cases}
&-\partial_t w_\eps-\mathcal H_\eps\Big(y,\partial_{y} w_\eps(t,y), \partial_{yy} w_\eps(t,y)\Big)= 0\; ,\quad (t,y)\in [0,T)\times \mathbb R \;,\\
&w_\eps(T,y)= -f(e^y)\;,\quad y\in \mathbb R \;,
\end{cases}\end{equation}
admits a unique solution $w_\eps$ in $C^{2+\nu}([0,T]\times\R)$ and
\beq\label{HmoinsHeps}
\sup_{\R^3}\big|\Hc-\Hc_\eps\big| & \leq & \eps
\enq
 for any $\eps>0$\;. 
\end{Proposition}
The proof of Proposition \ref{Approx-Hamiltonian} consists in an approximation by regularization of the original Hamiltonian $\Hc$. As it is quite technical we postpone this proof to the appendix.\\ 

\noindent We are now able to describe almost optimal contracts and related almost optimal efforts using the functions $w_\eps$ given by Proposition \ref{Approx-Hamiltonian}.

\begin{Theorem}\label{thm:regulator}  Suppose that  \textbf{(H')} holds. 
 For any $\eps>0$, 
the tax policy $\xi_\eps$ given by
\[
\xi_\eps= \tilde R-\int_0^T\big( g(X_t^{\mu}, Z_t^\eps)+\frac{1}{2}\sigma^2\gamma |Z^\eps_t|^2+ Z_t^\eps (\lambda-\mu(X_t^{\mu})) \big)dt+\int_0^T \frac{Z_t^\eps}{X_t^{\mu}}dX_t^{\mu} \;,
\]
where 
\beq\label{eqZstar}
Z^\eps_t &=&- \frac{\partial_x w_\eps(t,\log(X_t^{\mu}))}{1+\gamma\sigma^2}\;, \quad t\in[0,T]\;,
\enq
is $2T\eps$-optimal for the regulator problem:
\beqs
V_R^P &\leq&  \mathbb E^{a^*(X^\mu,Z^\eps)}\big[\xi_\eps-f(X_T^{\mu})\big]+2T\eps \;.
\enqs

\end{Theorem}

\begin{proof}

 We fix some control $Z\in\Zc$ and we apply It\^o's formula to the process $\big(Y^{\tilde R,Z,\mu}_t+ w_\eps(t,\log(X_t^{\mu}))\big)_{t\in[0,T]}$ 
\beqs
 Y_T^{\tilde R,Z,\mu}+w_\eps(T,\log(X_T^{\mu}))&= & \tilde R+ w_\eps(0,\log(X_0))\\
&&+\int_0^T\Big( \partial_t  w_\eps(s,\log(X^\mu_s))\\
 & &  +(\lambda-{\sigma^2\over 2}-\mu(X^\mu_s)- a^*(X^\mu_s,Z_s))\partial_{x}  w_\eps(s,\log(X^\mu_s))\\
&&+ p(X^\mu_s) a^*(X_s^\mu,Z_s) X^\mu_s-k( a^*(X^\mu_s,Z_s))-\frac{ \sigma^2}{2}\gamma |Z_s|^2\\
&&+  \frac{\sigma^2}{2} \partial_{xx}  w_\eps (s,\log(X^\mu_s)) \Big)ds\\
&&+\sigma\int_0^T\left(\partial_x    w_\eps(s,\log(X^\mu_s))+Z_s\right) dW_s^* \;,
\enqs 
where $W^*$ stands for $W^{ a^*(X^\mu,Z)}$. Since $w_\eps\in C^{2+\nu}([0,T]\times\R)$ and $Z\in\Zc$ we get 
\begin{multline*}
\E^{a^*(X^\mu,Z)}\left[Y_T^{\tilde R,Z,\mu}+ w_\eps(T,\log(X_T^{\mu}))\right] ~\leq ~   \tilde R+ w_\eps(0,\log(X_0)) \\
+\int_0^T\E^{a^*(X^\mu,Z)}\Big[\Big(\partial_t w_\eps+\Hc\Big(.,\partial_{y} w_\eps, \partial_{yy} w_\eps\Big)\Big)(s,\log(X^\mu_s))\Big]ds\;.
\end{multline*}
From \eqref{HmoinsHeps} we get
\begin{multline*}
\E^{a^*(X^\mu,Z)}\left[Y_T^{\tilde R,Z,\mu}+ w_\eps(T,\log(X_T^{\mu}))\right] ~\leq ~   \tilde R+ w_\eps(0,\log(X_0)) +T\eps\\
+\int_0^T\E^{a^*(X^\mu,Z)}\Big[\Big(\partial_t w_\eps+\Hc_\eps\Big(.,\partial_{y} w_\eps, \partial_{yy} w_\eps\Big)\Big)(s,\log(X^\mu_s))\Big]ds\;,
\end{multline*}
and since $w_\eps$ is solution to 
\eqref{pde:w-eps}, we get 
\beqs
\E^{a^*(X^\mu,Z)}\left[Y_T^{\tilde R,Z,\mu}-f(X_T^{\mu})\right] &\leq & \tilde R+ w_\eps(0,\log(X_0))+T\eps\;.
\enqs

Since $Z$ is arbitrarily chosen in $\Zc$ we get
\beq\label{dominationVRP}
V_R^P &\leq&  \tilde R+ w_\eps(0,\log(X_0))+T\eps\;.
\enq
We now take 
 $Z=Z^\eps$ where $Z^\eps$ is given by \eqref{eqZstar}. 
We now notice that $Z^\eps\in \mathcal Z$ since $Z^\eps$ is bounded, and by definition of $Z^\eps$ we have
\beqs
&& \Hc\Big(\log(X^\mu),\partial_{y} w_\eps(.,\log(X^\mu)), \partial_{yy} w_\eps(.,\log(X^\mu))\Big) \\
&= & X^\mu p(X^\mu )a^*(X^\mu ,Z^\eps)-{a^*(X^\mu ,Z^\eps)^2\over 2}-\frac{ \sigma^2}{2}\gamma \big|Z^\eps\big|^2 +\frac{\sigma^2}{2}\partial_{yy} w_\eps(.,\log(X^\mu))  \\
&&+ \quad  \big(\lambda-{\sigma^2\over 2}- \mu(X^\mu )-a^*(X^\mu ,Z^\eps)\big)\partial_y w_\eps(.,\log(X^\mu))  
\enqs
for any $[0,T]$. A straightforward application of It\^o's formula  and Girsanov Theorem give
\beqs
\mathbb E^{ a^*(X^\mu,Z^\eps)}[Y_T^{\tilde R,Z^\eps,\mu}-f(X_T^{\mu})] &=&  \tilde R+ w_\eps(0,\log(X_0))\\
 & &+ \int_0^T\E^{ a^*(X^\mu,Z^\eps)}\Big[\Big(\partial_t w_\eps+\Hc\Big(.,\partial_{y} w_\eps, \partial_{yy} w_\eps\Big)\Big)(s,\log(X^\mu_s))\Big]ds \;.
 \enqs
 From Propositions \eqref{pde:w-eps} and \eqref {HmoinsHeps} we get
\beqs
\mathbb E^{ a^*(X^\mu,Z^\eps)}[Y_T^{\tilde R,Z^\eps,\mu}-f(X_T^{\mu})]   & \geq & \tilde R+ w_\eps(0,\log(X_0))-T\eps\;.
\enqs
Hence, we get from \eqref{dominationVRP}
\beqs
 V_R^P & \leq & \mathbb E^{ a^*(X^\mu,Z^\eps)}[Y_T^{\tilde R,Z^\eps,\mu}-f(X_T^{\mu})]+2T\eps \;.
 \enqs
Therefore, we get $\xi_\eps= Y_T^{\tilde R,Z^\eps,\mu}$ is a $2T\eps$-optimal policy for the regulator.

\end{proof}

\subsection{Extension to the logistic equation and continuous cost}

 We consider in this section an approximation method to build a sequence of almost optimal taxes in the case the classical logistic dynamic for SDE \eqref{sdeprimal}, i.e. $\mu(x)=x$. More precisely, we introduce a sequence of approximated models from which we derive almost optimal strategy from the previous section. We show that this sequence remains almost optimal for the logistic model. We also weaken the assumption \textbf{(H')} $(ii)$ as follows. 

\vspace{1mm}

\noindent \textbf{(H$_f$')} The function $f$ is bounded and continuous on $\mathbb R$.\\

We introduce the sequence of mollifiers $\rho_n:~\R\rightarrow\R$, $n\geq 1$, defined by
 \beqs
 \rho_n(x) & := & {n\rho(nx)\over \int_\R\rho(u)du}\;, \qquad x\in\R\;,
 \enqs
 where the function $\rho:~\R\rightarrow\R$ is defined by
\begin{equation*}
\rho (x)  := \exp\big({-1\over 1-|x|^2}\big) \1_{|x] < 1}\;.
\end{equation*}
We then define the functions $f_n$, $n\geq 1$, by
\beqs
f_n(x) &: = & \int_\R f(y)\rho_n(x-y)dy\;,\quad x\in\R\;.
\enqs
From classical results, we know that $f_n$ satisfies $\textbf{(H')}(ii)$ for all $n\geq 1$ and $f_n$ converges to $f$ as $n$ goes to infinity uniformly on every compact subset of $\R$.

\noindent We also define the functions $\mu_n:~ \R\rightarrow\R$, $n\geq 1$, by
\beq\label{mundef}
\mu_n(x) & := & x\big( \Theta(x+e^n+1)-\Theta(x-(e^n+1)) \big)\;,\quad x\in\R\;,
\enq
where the function $\Theta:\R\rightarrow\R$ is given by 
\beq\label{defTheta}
\Theta(u) & := & {\int_{-\infty}^u\rho(r)dr\over \int_{\R}\rho(r)dr}\;,\quad u\in\R\;. 
\enq
We then notice that $\mu_n$ satisisfes \textbf{(H')} $(i)$ and
\beqs
\mu_n(x) & = & x\;,\quad x\in [-e^n,e^n]\;,
\enqs 
for $n\geq 1$.
We first have the following preliminary result on the convergence  of $X^{\mu_n}$ to $X$.
\begin{Lemma}\label{Lem-speed-cvXnX}
There exists a constant $C$ such that
\beqs
\sup_{t\in[0,T]}\E\Big[\big|X_t^{\mu_n}-X_t^{}\big|^2\Big] & \leq & C \exp\Big( -{n^2\over4\sigma^2} \Big)  \;,
\enqs
for all $n\geq 1$.
\end{Lemma} 
\begin{proof}
We define the sequence of stopping times $(\tau_n)_{n\geq1}$ by 
\beq\label{stopping-time}
\tau_{n} &:= & \inf \{t\in[0,T],\; X_t\geq e^n\}\;,\quad n\geq 1 \;.
\enq
We then notice that 
\beqs
X_{t\wedge \tau_n}&=&x+\int_0^{t\wedge \tau_n} X_{s}(\lambda-X_{s})ds+\int_0^{t\wedge \tau_n} \sigma X_{s}dW_s\\
&=& x+\int_0^{t} \mathbf 1_{s\leq \tau_n}X_{s}(\lambda -\mu_n(X_{s}))ds+\int_0^{t} \sigma \mathbf 1_{s\leq \tau_n}X_{s} dW_s \;,
\enqs
for all $t\in[0,T]$.
Therefore $(X_{t\wedge \tau_n})_{t\in [0,T]}$ and  $(X_{t\wedge \tau_n}^{\mu_n})_{t\in [0,T]}$ satisfy the same SDE with random and Lipschitz coefficients. By strong uniqueness,  we have
\beqs\label{egalitearret}
X_{t\wedge \tau_n} &=& X_{t\wedge \tau_n}^{\mu_n}\;,\quad   t\in [0,T] \;. 
 \enqs
Which implies
\[
 \sup_{t\in[0,T]}\E\Big[\big|X_t^{\mu_n}-X_t^{}\big|^2\Big] = \sup_{t\in[0,T]}\E\Big[\big|X_t^{\mu_n}-X_t^{}\big|^2\mathds{1}_{\tau_n\leq t}\Big]\;. 
 \]
Hence, by using Cauchy-Schwarz Inequality, we have
\beq \label{Markov}
\sup_{t\in[0,T]}\E\Big[\big|X_t^{\mu_n}-X_t^{}\big|^2\Big] & \leq & \sup_{t\in[0,T]}\E\Big[\big|X_t^{\mu_n}-X_t^{}\big|^4\Big]^{1\over 2}\P\big({\tau_n\leq T}\big)^{1\over 2} \;,
 \enq
 for all $n\geq 1$.
 
 We then notice that 
 \begin{equation}\label{nondecXn}
 X^{\mu_0}  \geq  X^{\mu_n} \geq X > 0
 \end{equation}
 for all $n\geq 1$. Indeed, by setting $\delta^n:=X^{\mu_n}-X^{}$,  we have
\[
\delta_t^n  =  \int_0^t b_s \delta_s^n ds+\int_0^t \sigma \delta_s^n d W_s+\int_0^tc_sds
\]
 where 
 \begin{equation*}
 b  := 
\left\{
\begin{array}{ccc}
\frac{ X^{\mu_n}(\lambda-\mu_n(X^{\mu_n})) -X(\lambda-\mu_n(X))}{\delta_n} & \mbox{ if } & X^{\mu_n}-X^{}\neq 0 \;,\\
 0 & \mbox{ if } & X^{\mu_n}-X = 0 \;,
 \end{array} 
 \right.
 \end{equation*}
is a bounded process since $x\mapsto x(\lambda-\mu_n(x))$ is Lipschitz continuous, and 
 \beqs
 c & = & X(\lambda-\mu_n(X)) -X(\lambda-X)
 \enqs
 is a nonnegative process since $\mu_n(x) \leq x$ for $x\in[0,+\infty)$. A straightforward computation shows that 
 \beqs
 \delta^n_t & = & R_t\int_0^t{c_s\over R_s}ds\;,\quad t\in[0,T]\;,
 \enqs
 where $R_t=\exp(\int_0^t(b_s-\sigma^2/2)ds+\sigma W_t)$ for $t\in[0,T]$. Since $c\geq 0$ we get $X^{\mu_n}\geq X$. The same argument applied to $\tilde \delta^n=X^{\mu_n}-X^{\mu_{n+1}}$ gives $X^{\mu_n}\geq X^{\mu_{n+1}}$.
 
 From \eqref{Markov} and \eqref{nondecXn}, there exists a constant $C$ such that 
 \beq\label{estim-intXnX}
 \sup_{t\in[0,T]}\E\Big[\big|X_t^{\mu_n}-X_t^{}\big|^2\Big] &\leq & C \P\big({\tau_n\leq T}\big)^{1\over 2}\;,\quad n\geq 1\;.
 \enq
 Still using $X^{\mu_0}\geq X>0$  and since $X^{\mu_0}$ is a geometric drifted brownian motion, we have
 \beqs
 \P\big({\tau_n\leq T}\big) & \leq & \P\big(\sup_{t\in[0,T]}W_t\geq {n-(\lambda-\sigma^2/2)T\over \sigma}\Big)  \;.
 \enqs
 Since $\sup_{t\in[0,T]}W_t$ has the same law as $|W_T|$, we get
 \begin{equation}\label{estimtaunT}
 \P({\tau_n\leq T}\big)  \leq  2\int_{{n-(\lambda-\sigma^2/2)T\over \sigma}}^{+\infty} e^{-y^2\over 2}{dy\over \sqrt{2\pi}}
  \leq  C\exp\big(-{n^2\over 2\sigma^2}\big) \;,\quad n\geq 1\;,
 \end{equation}
 and we get the result from \eqref{estim-intXnX}.
\end{proof}

We then define the function $V^P_{R,n}$ as the optimal value of the regulator problem in the model with coefficients $\mu_n$ and $f_n$ in place of $\mu$ and $f$ respectively
\beqs
V^P_{R,n} & = & \sup_{\xi \in \Cc^{\mu_n}_R} \mathbb E^{a^*(X^{\mu_n},Z)}[\xi-f_n(X_T^{\mu_n})] \;. 
\enqs
Since $\mu_n$ and $f_n$ satisfy Assumptions \textbf{(H')} $(i)$ and \textbf{(H')} $(ii)$ respectively for all $n\geq1$, we get, from Theorem \ref{thm:regulator}, a sequence of bounded processes $(Z^{\eps,n})_{n\geq1}$ such that
\beqs
V^P_{R,n} & \leq  & \mathbb E^{a^*(X^{\mu_n},Z^{\eps,n})}[Y^{\tilde R,Z^{\eps,n},\mu_n}_T-f_n(X_T^{\mu_n})]+2T\eps\;,\quad n\geq 1\;.
\enqs
We introduce the control $\tilde Z^{\eps,n}$ defined by
\beqs
\tilde Z^{\eps,n}_t & := & Z^{\eps,n}_t\mathds{1}_{[0,\tau_n]}(t)\;,\quad t\in[0,T]\;,
\enqs
where  the stopping time $\tau_n$  is defined in \eqref{stopping-time}, and we denote by $\tilde \xi_{\eps,n}$ the related contract
\[
\tilde \xi_{\eps,n}  =  Y_T^{R,\tilde Z^{\eps,n}}\;.
\]

We then have the following almost optimality result.
\begin{Theorem}\label{thm:logistic}
Suppose that  \textbf{(H$_f$')} and \textbf{(H')} $(iii)$ hold.  Then $V_R^P<+\infty$ and we have

\[
\limsup_{n\rightarrow+\infty} \Big(V^P_{R}- \mathbb E^{\mathbb P^{ a^*(X^{},\tilde Z^{\eps,n})}}\big[\tilde \xi_{\eps,n}-f(X_T^{})\big] \Big) \leq  2T\eps
\]
for any $\eps>0$.

\end{Theorem}
\begin{proof} We proceed in four steps.

\vspace{2mm}

\ni \textit{Step 1. The optimal value $V_R^P$ is finite.}

\ni From \eqref{value:reg:ximu2}  and the dynamics of $Y^{\tilde R,Z}$ we have
\beqs
V_R^P &\leq& \sup_{ Z\in \mathcal Z} \mathbb E^{a^*(X,Z)}[{\tilde R}+T\underline M P-f(X_T^{})] ~<~+\infty\;.
\enqs
Since $f$ is bounded, we get $V_R^P<+\infty$.
\vspace{2mm}

\ni \textit{Step 2. Comparison of $\Xi^\mu$ and $\Xi^{\mu'}$.}

\ni We fix two functions $\mu,\mu'$ satisfying $\textbf{(H0)}$ and we show that $\Xi^\mu=\Xi^{\mu'}$where $\Xi^\mu$ and $\Xi^{\mu'}$ are defined by \eqref{defXimu}. Let $\xi=Y_T^{Y_0,Z,\mu}\in\Xi^\mu$. Then, we have by definition
\beqs
 \mathbb E[\exp(\gamma' |Y_T^{Y_0,Z,\mu}|)]  & < & +\infty \;,
 \enqs
for some $\gamma'>2\gamma $, with
 \[
Y_T^{Y_0,Z,\mu} = Y_0-\int_0^T \big(\frac{| a^*(X^\mu_s,Z_s)|^2}2- a^*(X^\mu_s,Z_s) (p(X^{\mu}_s)X_s^{\mu}+Z_s)+\frac{ \sigma^2}{2}\gamma |Z_s|^2  \big)ds +\int_0^T \sigma Z_s dW_s \;.
\]
Since the optimal effort $a^*$ is bounded and Assumption \textbf{(H')} $(iii)$ holds, there exists a positive constant $C$ such that
\beqs
\mathbb E \big[\exp(\gamma' |Y_T^{Y_0,Z,\mu'}|) \big] &= &\mathbb E\Big[e^{\gamma' \big|Y_0-\int_0^T \big(\frac{| a^*(X^{\mu'}_s,Z_s)|^2}2- a^*(X^{\mu'}_s,Z_s) (p(X^{\mu'}_s)X_s^{\mu'}+Z_s)+\frac{ \sigma^2}{2}\gamma |Z_s|^2  \big)ds+\int_0^T \sigma Z_s dW_s\big|}\Big] \\
&&\\
&\leq &C\mathbb E\Big[e^{\gamma' \big( \big|Y_0-\int_0^T  \big(\frac{| a^*(X^{\mu}_s,Z_s)|^2}2- a^*(X^{\mu'}_s,Z_s)Z_s + \frac{\sigma^2}{2}\gamma |Z_s|^2  \big)ds+\int_0^T \sigma Z_s dW_s\big|\big)}\Big] \\
&&\\
&\leq &  C\mathbb E\big[e^{\gamma' |Y_T^{Y_0,Z,\mu}|} e^{\delta^*_T}\big] \;,
\enqs
where $\delta^*_T={\gamma' \int_0^T | a^*(X^{\mu'}_s,Z_s)- a^*(X^{\mu}_s,Z_s)| |Z_s|ds}\;.$
We then notice  that
\beqs
| a^*(X',z)-  a^*(X,Z)||z| &=& \big|(\overline M \wedge (p(x')x'+z)\vee(-\underline M)) -(\overline M\wedge (p(x)x+z)\vee(-\underline M))\big| |z| \\
 & \leq & (\overline M+\underline M)(P+\overline M+\underline M)
\enqs
for any $z\in \mathbb R$. Hence, we get
\[
\mathbb E[\exp(\gamma' |Y_T^{Y_0,Z,\mu'}|)] \leq   Ce^{\gamma'T(\overline M+\underline M)(P+\overline M+\underline M)}\mathbb E\big[e^{\gamma' |Y_T^{Y_0,Z,\mu}|}\big] \;.
\]

 Hence, we get $\xi\in \Xi^{\mu}$.  We then write $ \Xi$ for $ \Xi^{\mu}$ in the sequel.

\vspace{2mm}

\ni \textit{Step 3. Convergence of the values for a given $Z$.}

\ni We fix $Z\in\Xi$. We then have
\begin{align*}
\nonumber 
\Delta_n(Z) & :=  \big| \mathbb E^{ a^*(X^{\mu_n},Z)}[Y_T^{\tilde R,Z,\mu_n}-f_n(X_T^{\mu_n})]   -\mathbb E^{ a^*(X,Z)}[Y_T^{\tilde R,Z}-f(X_T^{})]  \big| \\
&\leq \Delta^n_1(Z)+\Delta^n_2(Z) \;, 
\end{align*}
with 
\beqs
\Delta^n_1(Z) &=& \mathbb E \Big[|H_T^{ a^*(X^{\mu_n},Z)}- H_T^{ a^*(X^{},Z)}| |Y_T^{\tilde R,Z}-f(X_T)| \Big] 
\enqs 
and 
\beqs
\Delta^n_2(Z) &=&  \mathbb E \Big[H_T^{ a^*(X^{\mu_n},Z)} |(Y_T^{\tilde R,Z} -f(X_T)) - (Y_T^{\tilde R,Z,\mu_n} - f_n(X_T^{\mu_n})) |\Big]\;.
\enqs

\noindent We now study the convergence of $\Delta_n^1$ and $\Delta_n^2$.

\vspace{2mm}

\ni \textit{Substep 3.1.  Convergence of $\Delta_1^n$.}

\noindent  By  Cauchy-Schwarz inequality we have
\beqs
\Delta_1^n(Z)&\leq  & \mathbb E[|H_T^{ a^*(X^{\mu_n},Z)}- H_T^{a^*(X^{},Z)}|^2]^{\frac12} \mathbb E[|Y_T^{\tilde R,Z}-f(X_T)|^2]^{\frac12} \;.
\enqs
From \eqref{integr:z} and \textbf{(H$_f$')}, $\mathbb E[|Y_T^{\tilde R,Z}-f(X_T)|^2]^{\frac12}$ is uniformly bounded w.r.t. $n$. The convergence of $\Delta_1^n(Z)$ remains to the convergence of 
\beqs
\tilde \Delta^n_1(Z) &=& \mathbb E[|H_T^{ a^*(X^{\mu_n},Z)}- H_T^{ a^*(X^{},Z)}|^2]\;.
\enqs
From the definition of $a^*$, for all $q\geq 1$, there exists a constant $C_q$ such that 
\beqs
\E[| a^*(X^{},Z)- a^*(X^{\mu_n},Z)|^q] & \leq & C_q \E[| a^*(X^{},Z)- a^*(X^{\mu_n},Z)|^2]\;,\quad n\geq 1\;.
\enqs
Therefore we have
\beqs
\sup_{t\in[0,T]}\E[| a^*(X^{}_t,Z_t)- a^*(X^{\mu_n}_t,Z_t)|^q] & \xrightarrow[n\rightarrow+\infty]{} & 0
\enqs
for all $q\geq 1$. From  Lemma \ref{Lem-speed-cvXnX} and Theorem 2.8.1 in \cite{Kry80} we get
$\tilde \Delta^n_1(Z)  \xrightarrow[n\rightarrow+\infty]{}0$.

\vspace{2mm}

\ni \textit{Substep 3.2.  Convergence of $\Delta_2^n$.}

\ni From Cauchy-Schwarz Inequality, there exists a positive constant $C>0$ such that
\beqs
\Delta^n_2(Z)
&\leq&  \mathbb E[|H_T^{  a^*(X^{\mu_n},Z)}|^2]^\frac12\mathbb E[|(Y_T^{\tilde R,Z}-f(X_T)) - (Y_T^{\tilde R,Z,\mu_n}-f_n(X_T^{\mu_n}))|^2]^\frac12\\
 & \leq & C\; \mathbb  E[|H_T^{  a^*(X^{\mu_n},Z)}|^2]^\frac12\Big( \mathbb E\big[\big|Y_T^{\tilde R,Z,\mu_n}- Y_T^{\tilde R,Z}\big|^2\big]^{1\over 2} +\mathbb E\big[\big|f(X_T)- f_n(X_T^{\mu_n})  \big|^2\big]^{1\over2}\Big)\;.
\enqs

\noindent First note that 
\beqs
\mathbb E[|H_T^{  a^*(X^{\mu_n},Z)}|^2]&= &\mathbb E[e^{-2\int_0^T a^*(X^{\mu_n}_s,Z_s)\sigma^{-1}d W_s -\int_0^T|a^*(X^{\mu_n}_s,Z_s)|^2\sigma^{-2}ds }]\\
&= &\mathbb E^{\mathbb Q}[e^{\int_0^T|a^*(X^{\mu_n}_s,Z_s)|^2\sigma^{-2}ds}] \;,
\enqs
with $d\mathbb Q/d\mathbb P = H_T^{2 a^*(X^{\mu_n},Z)}$.
Since $|a^*(X^{\mu_n},Z)|$ is bounded by $\underline M  \vee \overline M$, we deduce that
\[\mathbb E[|H_T^{a^*(X^{\mu_n},Z)}|^2]\leq e^{ \frac{T (\underline M\vee \overline M)^2}{\sigma^2}}.\]
We then have
\[
\mathbb E\big[\big|f(X_T)- f_n(X_T^{\mu_n})  \big|^2\big]^{1\over2}  \leq  C \; \mathbb E\big[\big|f(X_T)- f_n(X_T)  \big|^2\big]^{1\over2}+ C \; \mathbb E\big[\big|f_n(X_T)- f_n(X_T^{\mu_n})  \big|^2\big]^{1\over2} \;.
\]
Since $f$ is continuous and bounded, we get from the dominated convergence Theorem 
\[
\mathbb E\big[\big|f(X_T)- f_n(X_T)  \big|^2\big]\xrightarrow[n\rightarrow+\infty]{}  0\;.
\]
Then from the definition of $f_n$ there exists a constant $L$ such that $f_n$ is $L$-Lipchitz continuous for all $n\geq 1$. Therefore, we get from Lemma \ref{Lem-speed-cvXnX} 
\beqs
\mathbb E\big[\big|f_n(X_T)- f_n(X_T^{\mu_n})  \big|^2\big] & \xrightarrow[n\rightarrow+\infty]{} & 0\;.
\enqs
Since $a^*$ is continuous and bounded, and $Z\in\Zc$, we get from Lemma \ref{Lem-speed-cvXnX} and the definition of  $Y^{\tilde R,Z,\mu_n}$ and $Y^{\tilde R,Z}$
\beqs
\mathbb E\big[\big|Y_T^{\tilde R,Z,\mu_n}- Y_T^{\tilde R,Z}\big|^2\big] & \xrightarrow[n\rightarrow+\infty]{} & 0\;.
\enqs
Hence we get $\lim_{n\rightarrow+\infty}\Delta^n_2(Z)=0$.
\vspace{2mm}

\ni \textit{Step 4. Almost optimality of $\tilde Z^{\eps,n}$.}

\ni We fix $\eta>0$ and $Z^\eta\in \Xi$ such that
\beqs
 V^P_R & \leq & \mathbb E^{ a^*(X,Z^\eta)}[Y_T^{\tilde R,Z^\eta}-f(X_T^{})]+\eta \;.
\enqs
By definition of $Z^{\eps,n}$, we get

\beqs
V^P_{R}- \mathbb E^{\mathbb P^{ a^*(X^{\mu_n},Z^{\eps,n})}}[\xi_{\eps,n}-f_n(X_T^{\mu^n})] & \leq & \eta+4T\eps+ \Delta_n(Z^\eta)\;, \quad n\geq 1\;.
\enqs
Sending $n$ to $\infty$, we get from Step 2
\beqs
\limsup_{n\rightarrow+\infty} V^P_{R}- \mathbb E^{\mathbb P^{ a^*(X^{\mu_n},Z^{\eps,n})}}[\xi_{\eps,n}-f_n(X_T^{\mu_n})] & \leq & 4T\eps + \eta
\enqs
for any $\eta>0$. Which implies 
\beq\label{estimlast}
\limsup_{n\rightarrow+\infty} V^P_{R}- \mathbb E^{\mathbb P^{ a^*(X^{\mu_n},Z^{\eps,n})}}[\xi_{\eps,n}-f_n(X_T^{\mu_n})] & \leq & 4T\eps\;. 
\enq
Since $f$ and $a^*$ are bounded, there exists a constant $C$ such that
\begin{multline*}
|\mathbb E^{\mathbb P^{ a^*(X^{\mu_n},Z^{\eps,n})}}[f_n(X_T^{\mu_n})] -\mathbb E^{\mathbb P^{ a^*(X^{},\tilde Z^{\eps,n})}}[f(X_T^{})]| \\
 \leq C\Big(\mathbb E[|H_T^{ a^*(X^{\mu_n},Z^{\eps,n})}- H_T^{ a^*(X^{}, Z^{\eps,n})}|^2]^{1\over 2}+\P(\tau_n\leq T)+ \mathbb E\big[\big|f(X_T)- f_n(X_T^{\mu_n})  \big|^2\big]^{1\over2} \Big)\;.
\end{multline*}
We therefore get from Step 3 and  \eqref{estimtaunT}
\beq\label{estim0}
\mathbb E^{\mathbb P^{ a^*(X^{\mu_n},Z^{\eps,n})}}[f_n(X_T^{\mu_n})] -\mathbb E^{\mathbb P^{ a^*(X^{},\tilde Z^{\eps,n})}}[f(X_T^{})] & \xrightarrow[n\rightarrow+\infty]{} & 0\;.
\enq
From \textbf{(H')} $(i)$ and the definition of $a^*$ we have
\begin{multline}\label{estim-Z-casse-pied1}
\mathbb E^{\mathbb P^{ a^*(X^{\mu_n},Z^{\eps,n})}}[Y_T^{R,Z^{\eps,n}}] -\mathbb E^{\mathbb P^{ a^*(X^{},\tilde Z^{\eps,n})}}[\tilde \xi_{\eps,n}]    \\ 
 \leq \mathbb E\Big[H^{{ a^*(X^{},\tilde Z^{\eps,n})}}_T\int_0^T\Big({ a^*(X^{}_s,\tilde Z^{\eps,n}_s)^2\over 2}+ a^*(X^{}_s,\tilde Z^{\eps,n}_s)p(X^{}_s)X^{}_s\Big)ds\Big] \\
\qquad \qquad -\mathbb E\Big[H^{{ a^*(X^{\mu_n},Z^{\eps,n})}}_T\int_0^T\Big({ a^*(X^{\mu_n}_s,Z^{\eps,n}_s)^2\over 2}+ a^*(X^{\mu_n}_s,Z^{\eps,n}_s)p(X^{\mu_n}_s)X^{\mu_n}_s\Big)ds\Big] \\
+{\gamma\sigma^2\over2}\E\Big[H_T^{{ a^*(X^{},\tilde Z^{\eps,n})}}\int_0^{T}|\tilde Z^{\eps,n}_s|^2ds-H_T^{{ a^*(X^{\mu_n},Z^{\eps,n})}}\int_0^{\tau_n}| Z^{\eps,n}_s|^2ds\Big]\;. 
\end{multline}
By definition of $\tau_n$ we have
\beq\label{estim-Z-casse-pied2}
\E\Big[H_T^{{ a^*(X^{},\tilde Z^{\eps,n})}}\int_0^{T}|\tilde Z^{\eps,n}_s|^2ds-H_T^{{ a^*(X^{\mu_n},Z^{\eps,n})}}\int_0^{\tau_n}| Z^{\eps,n}_s|^2ds\Big] & = & 0\;,
\enq
and
\beq \label{estim-Z-casse-pied3}
\mathbb E\Big[H^{{ a^*(X^{},\tilde Z^{\eps,n})}}_T\int_0^{T\wedge\tau_n}\Big({ a^*(X^{}_s,\tilde Z^{\eps,n}_s)^2\over 2}+ a^*(X^{}_s,\tilde Z^{\eps,n}_s)p(X^{}_s)X^{}_s\Big)ds\Big] &  &\\ \nonumber
-\mathbb E\Big[H^{{ a^*(X^{\mu_n},Z^{\eps,n})}}_T\int_0^{T\wedge \tau_n}\Big({ a^*(X^{\mu_n}_s,Z^{\eps,n}_s)^2\over 2}+ a^*(X^{\mu_n}_s,Z^{\eps,n}_s)p(X^{\mu_n}_s)X^{\mu_n}_s\Big)ds \Big] &  =& 0\;.\quad 
\enq
Since $a^*$ and $x\mapsto p(x)x$ are bounded, we get from  \eqref{estimtaunT}

\beq\label{estim-Z-casse-pied4}
\mathbb E\Big[H^{{ a^*(X^{},\tilde Z^{\eps,n})}}_T\int_{T\wedge\tau_n}^T\Big({ a^*(X^{}_s,\tilde Z^{\eps,n}_s)^2\over 2}+ a^*(X^{}_s,\tilde Z^{\eps,n}_s)p(X^{}_s)X^{}_s\Big)ds\Big] &  &\\
-\mathbb E\Big[H^{{ a^*(X^{\mu_n},Z^{\eps,n})}}_T\int_{T\wedge \tau_n}^T\Big({ a^*(X^{\mu_n}_s,Z^{\eps,n}_s)^2\over 2}+ a^*(X^{\mu_n}_s,Z^{\eps,n}_s)p(X^{\mu_n}_s)X^{\mu_n}_s\Big)ds \Big] &  \xrightarrow[n\rightarrow+\infty]{} & 0\;.
\qquad  \nonumber
\enq
Therefore we get from \eqref{estim0}, \eqref{estim-Z-casse-pied1}, \eqref{estim-Z-casse-pied2}, \eqref{estim-Z-casse-pied3}  and  \eqref{estim-Z-casse-pied4} 
\beqs
\limsup_{n\rightarrow+\infty} \mathbb E^{\mathbb P^{ a^*(X^{\mu_n},Z^{\eps,n})}}[Y_T^{R,Z^{\eps,n}}-f_n(X_T^{\mu_n})] -\mathbb E^{\mathbb P^{ a^*(X^{},\tilde Z^{\eps,n})}}[\tilde \xi_{\eps,n}-f(X_T^{})]  & \leq & 0\;.
\enqs
This last inequality with \eqref{estimlast} give the result.

\end{proof}

\subsection{Applications and economical interpretations}

\subsubsection{Non-regulated case and reservation utility}\label{application:reservation}
In this part, we provide a way to monitor the activities of the natural resource manager without penalizing him by choosing a relevant reservation utility $R$. The natural way is to consider the problem of the regulator without regulation policy
$$
\overline V_A  :=  \sup_{\alpha \in \Ac} \E^{\alpha} \Big[ - \exp \Big( -\gamma \big( \int_0^T p(X_s^{\mu})\alpha_sX^\mu_sds - \int_0^T \frac{|\alpha_s|^2}2 ds \big) \Big) \Big]\;.
$$

\noindent This problem can be solved explicitly by direct computations. If the regulator chooses $R=\overline V_A$ then any admissible tax $\xi$ will satisfy $V_A(\xi)\geq \overline V_A$. In other words, the choice of $R$ ensures a non-punitive regulation policy. 

\subsubsection{Numerical examples}

We now give some numerical results to illustrate our theoretical results. For that we consider $\mu_n$ given by \eqref{mundef}, $p(x)=px^{-1}$ with $p>0$ and $f(x)= (c-\frac{c}{\beta}x) \mathbf 1_{x<\beta}$
where $c$ is the cost of the resource for the regulator and $\beta$ is the target size of the population.

We use the following parameters:  
$\gamma=0.1$, $\lambda=1.2$, $\sigma=0.1$, $P=1$, $T=1$, $\beta=0.9$, $c=3$, $\underline M = \overline M = 10$, $n=100$, $\varepsilon = 0.01$ with $X_0= 1.2$. We use an approximation grid of $2000$ points for the space and $5000$ points for the time. In our case $R=\overline V_A = - \exp(- \gamma P^2 T)$.

\begin{figure}[h]
    \begin{minipage}[c]{.46\linewidth}
        \centering 
\includegraphics[width=8.5cm]{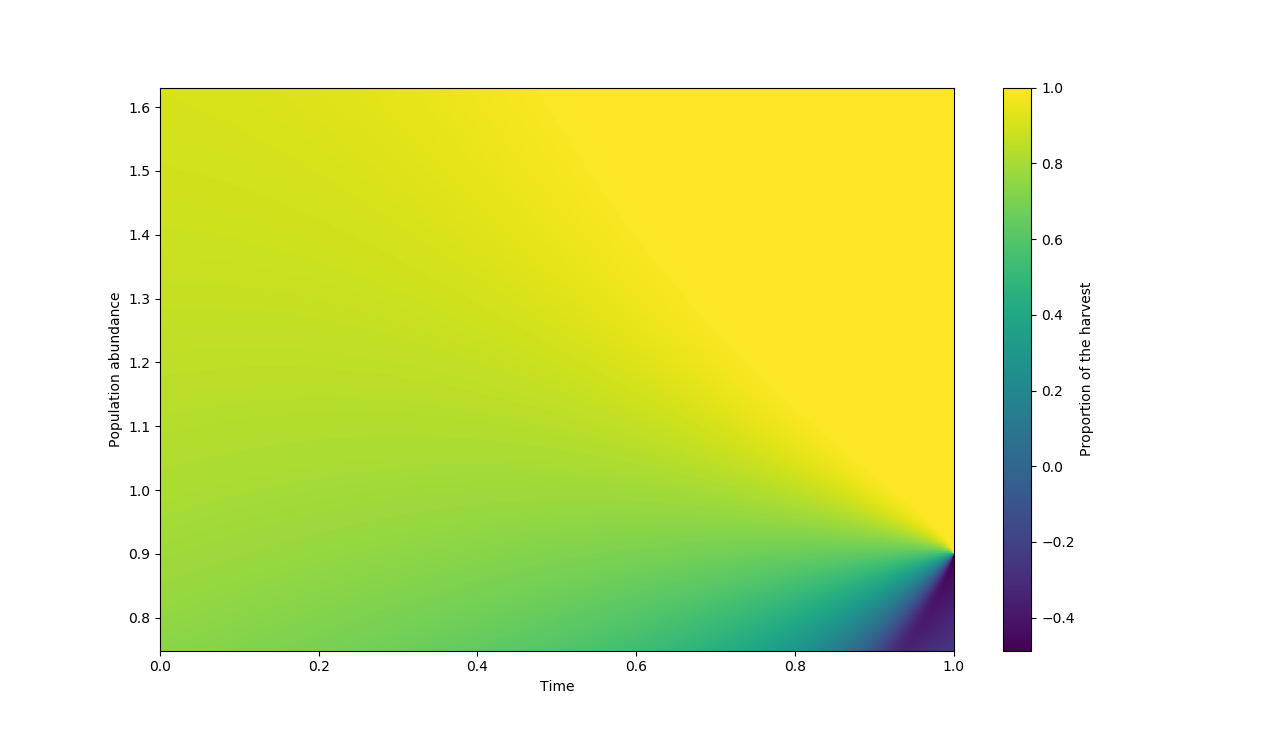}
\caption{\label{couleur 1} The optimal harvest rate w.r.t.  the time $t$ and the population abundance $X_t$.}
    \end{minipage}
    \hfill%
    \begin{minipage}[c]{.46\linewidth}
        \centering \vspace{-0.4cm}
\includegraphics[width=8.5cm]{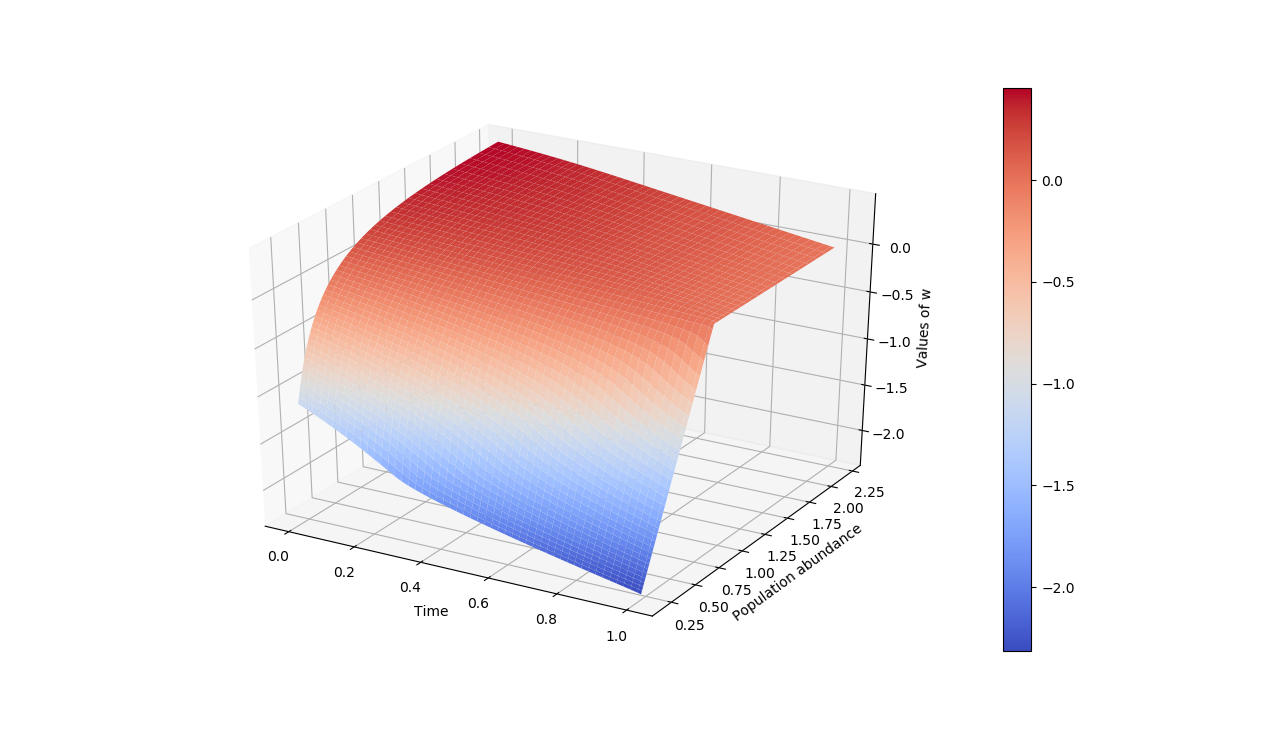}
\caption{\label{Fonction-valeur-3D} The value function $w_\varepsilon$  w.r.t. the time and the population abundance.}

    \end{minipage}
\end{figure}
Figure \ref{couleur 1} shows  the Agent harvests moderately   at the beginning and the rate is increasing w.r.t. the abundance population. On the contrary, for times close to the maturity the strategy depends on the abundance of the resource. Indeed, for an abundance below the target $\beta$, \textit{i.e.} $X^\alpha_t < 0.9$, the Agent renews the population, and for an abundance higher than $\beta$, \textit{i.e.} $X^\alpha_t \geq 0.9$, the Agent harvests. Moreover, the lowest is the abundance, the most the Agent renews the population, and the highest is the abundance, the most the Agent harvests the population.\\
 
Figure \ref{Fonction-valeur-3D} presents the graph of the value function $w_\varepsilon$ of the Principal. The function $w_\eps$ is increasing w.r.t. the population abundance. This property is expected in view of the Principal optimization problem. We also remark that the value function $w_\varepsilon$ is decreasing w.r.t. the time to maturity. Indeed, the longer is the time to maturity, the best it is for the Agent and the Principal since the resource has more time to regenerates itself.\\

\begin{figure}[h]
    \begin{minipage}[c]{.46\linewidth}
\hspace{-1cm}
\includegraphics[width=8.5cm]{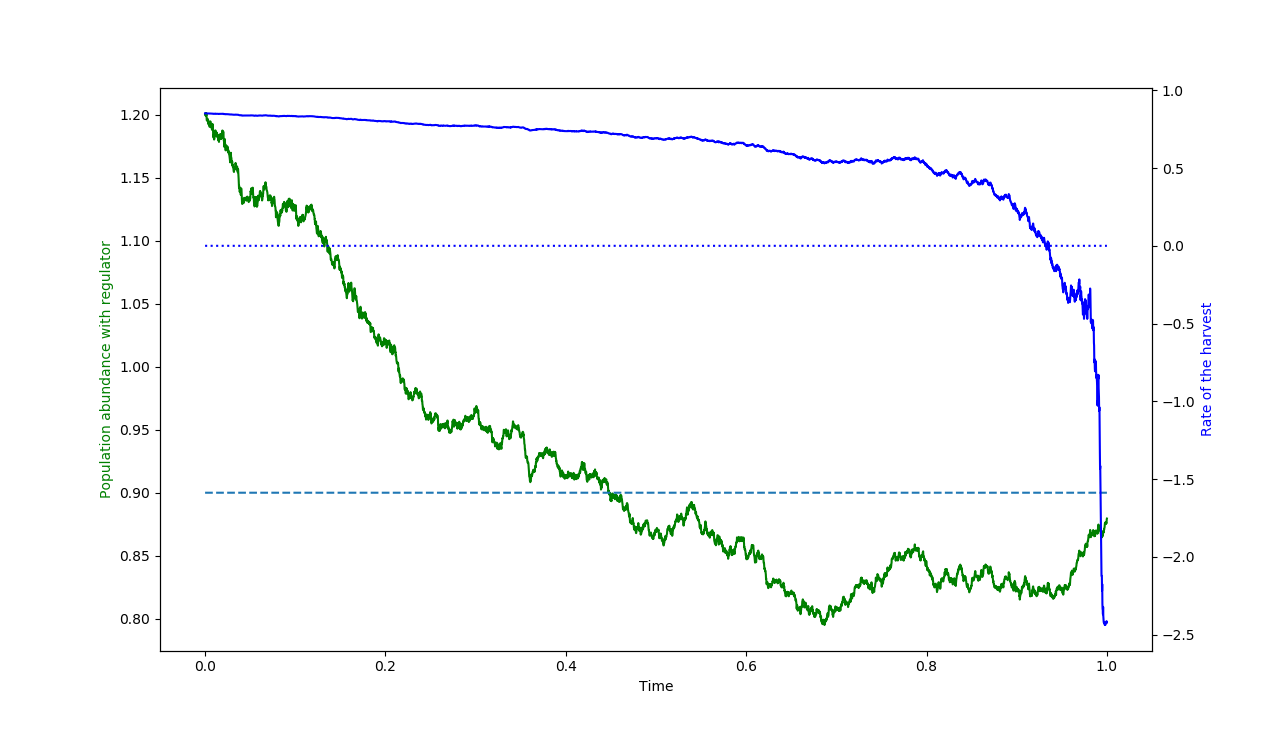}
\caption{\label{trajectoire 1} A trajectory of the optimally controlled population abundance (green curve, $y$-axis on the left) and the associated optimal harvest rate (blue curve, $y$-axis on the right) w.r.t. the time ($x$-axis). The dotted line corresponds to $\alpha=0$, and the dashed line corresponds to $X_t^{\alpha^*}=\beta$. }
    \end{minipage}
    \hfill%
  \begin{minipage}[c]{.46\linewidth}
 \vspace{-1.3cm}  \hspace{-1cm} 
\includegraphics[width=8.5cm]{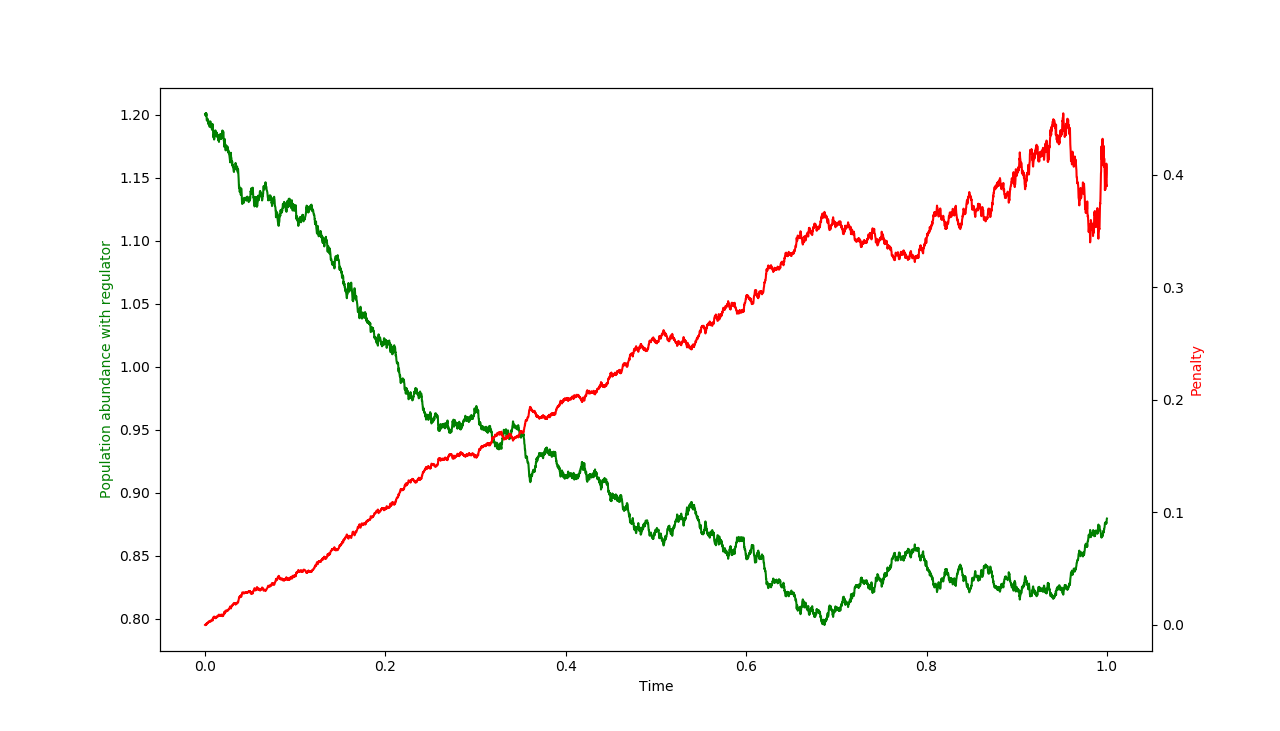}
\caption{\label{penalite1} The evolution of the penalty ($y$-axis on the right, red curve) and the population abundance ($y$-axis on the left, green curve)  w.r.t. the time ($x$-axis).}
    \end{minipage}
\end{figure}

Figure \ref{trajectoire 1} shows that the  Agent harvests with an important rate at the beginning : $\alpha_t$ is around $0.6$. As he get closer to the matirity, the Agent slows down the harvest and then renew the resource. This can be interpreted as follows.

The Agent harvests with a high rate and do note care about the tax at maturity since  the population as  has time to regenerate itself.  

Getting closer to the maturity, the Agent take into acount the tax and slows down the harvest. When very close to maturity, $t\approx0.93$, the Agent renews the population to ensure an abundance  close to the target $\beta = 0.9$ to limit the tax. This shows that the incentive policy is efficient. \\

Figure \ref{penalite1} presents the forecast of the penalty (\textit{i.e.} $Y^{\tilde R, Z^\varepsilon, n}$) in red, and the abundance population in green. 
We notice these two quantities evolve in opposite ways: for high values of the abundance, the expected tax is low, and for low abundance the penalty becomes greater.

\noindent We now study the sensitivity of the incentive policy w.r.t. the target $\beta$ (see Figure \ref{Trajectoire moyenne beta}) and w.r.t. the renewal cost $c$ (see Figure \ref{Trajectoire moyenne c}).
\begin{figure}[h]
    \begin{minipage}[c]{.4\linewidth}
	\hspace{-1cm}\includegraphics[scale=0.28]{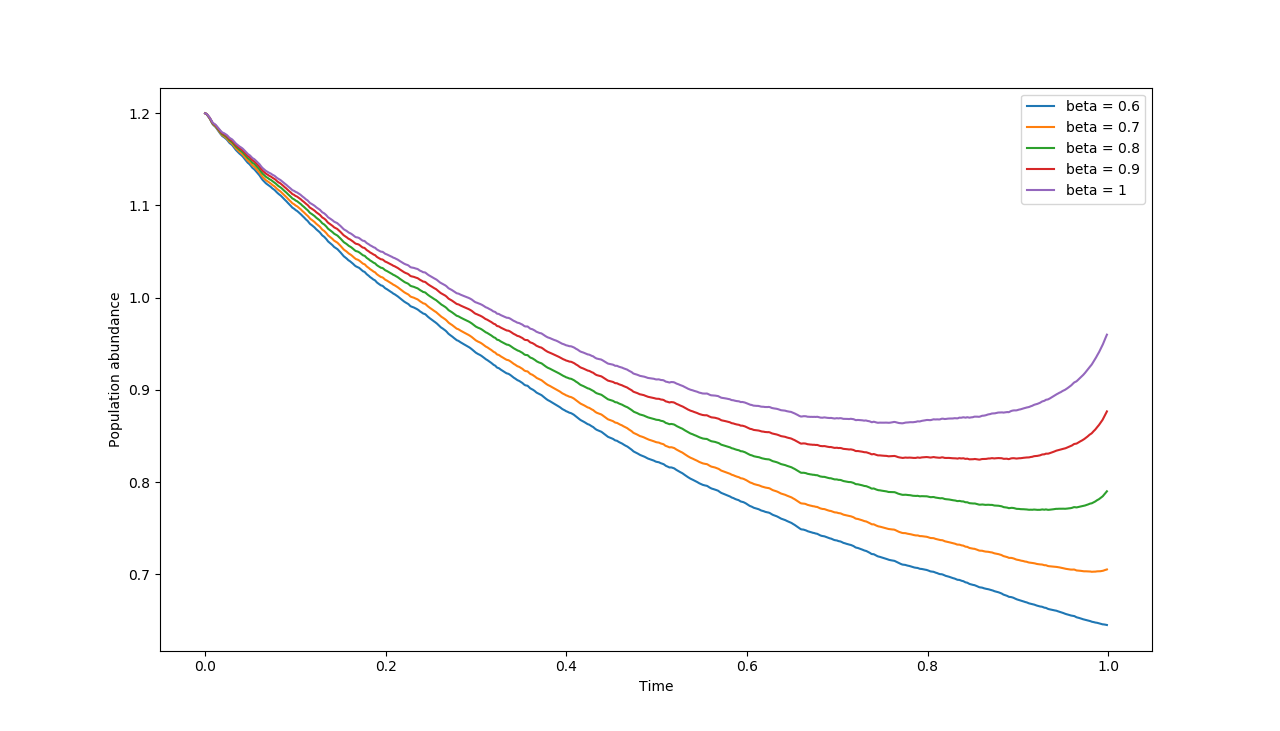}
	\caption{\label{Trajectoire moyenne beta} Evolution in mean of the population abundance w.r.t. the time for different values of $\beta$.}
    \end{minipage}
    \hfill
  \begin{minipage}[c]{.4\linewidth}
	\hspace{-1.5cm}\includegraphics[scale=0.28]{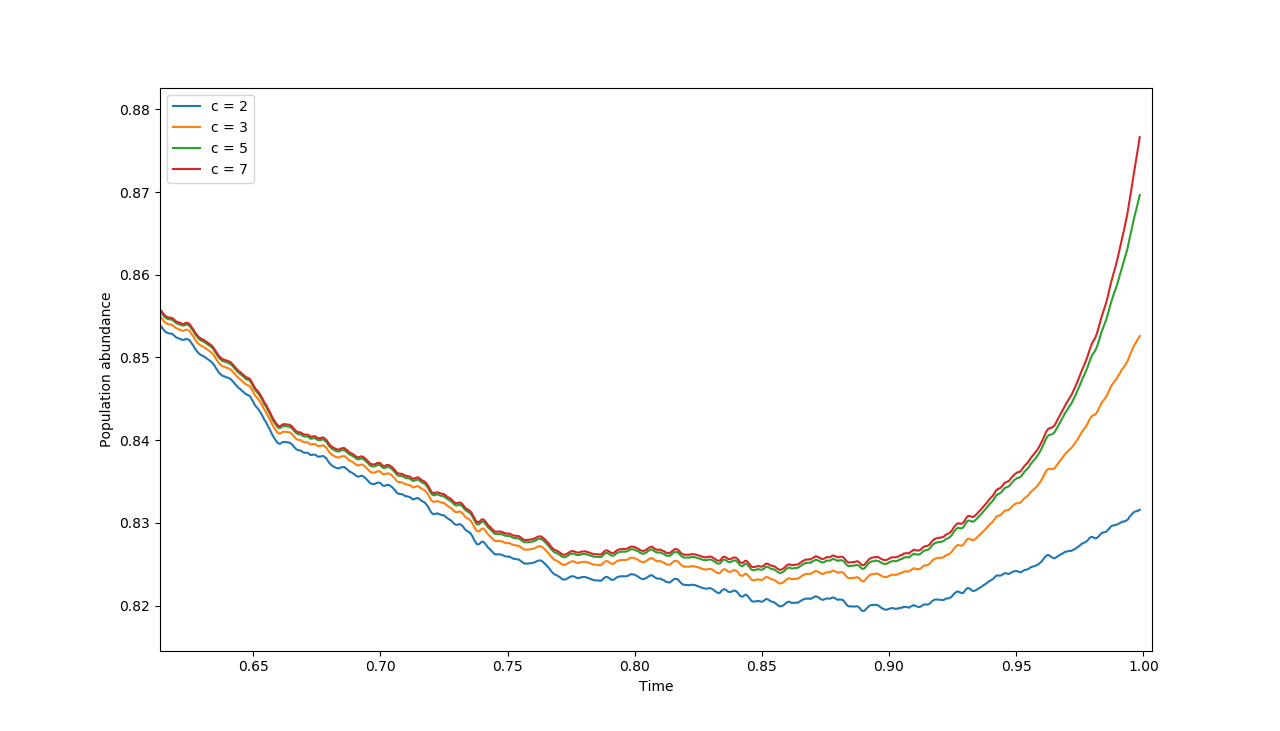}
	\caption{\label{Trajectoire moyenne c}  Evolution in mean of the population abundance w.r.t. the time for different values of $c$.}
    \end{minipage}
\end{figure}

Figure \ref{Trajectoire moyenne beta} presents the evolution in mean of the abundance w.r.t. time for several values of $\beta$. The mean is approximated by the empirical mean over 1000 trajectories. We remark that at each time the mean of the population abundance is more important as $\beta$ is larger. This shows that the choice of $\beta$ influences the behavior of the resource manager: the most is important $\beta$ the least the Agent harvests. We also notice that for each value of $\beta$, the mean terminal value reaches the target, which also shows the incentive effect of the parameter $\beta$.\\ 

 Figure \ref{Trajectoire moyenne c} shows the evolution of the mean of the resource abundance w.r.t. time for several values of the costs parameter $c$. The mean is approximated by the empirical mean over  for 1000 trajectories. We remark that the population abundance is nondecreasing w.r.t. $c$. In particular, the highest the penalty is, the most the Agent is concerned, through the incentive policy, by the size of the population at the end.\\

\noindent We now compare the situation for which the Agent can renew the population abundance (that is $\underline M > 0$) with the situation for which the Agent can only harvest (that is $\underline M=0$).
\begin{figure}[h]
\centering
\includegraphics[width=9cm]{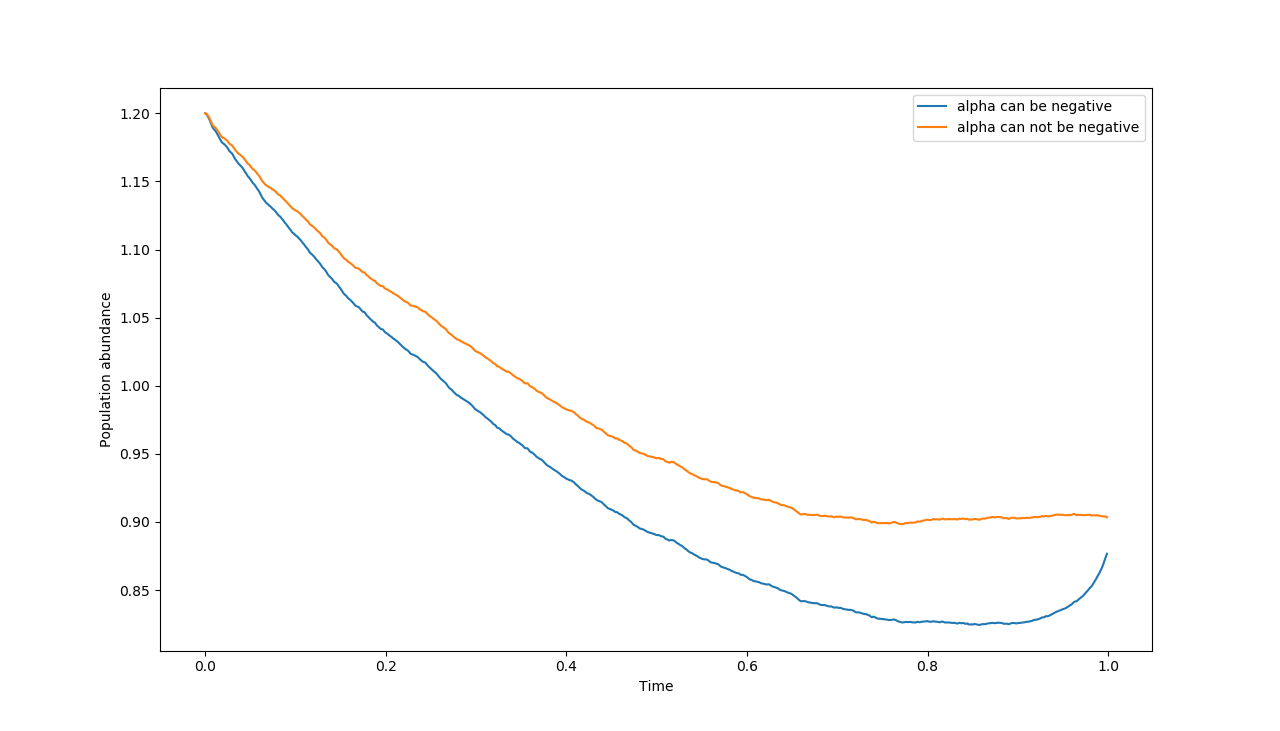}
\caption{\label{Trajectoire moyenne sans renouvellement} Evolution in mean of the population abundance w.r.t. the time when the Agent can (blue curve) and cannot (orange curve) renew the population abundance.}
\end{figure}
 In Figure \ref{Trajectoire moyenne sans renouvellement} the mean is approximated by the empirical mean over 1000 trajectories.  We remark that at each time the population abundance is  more important in mean if the Agent cannot renew the resource. 
 Indeed, if  the resource is not renewable, the Agent reduces his harvesting rate in prevision of the terminal tax. On the contrary, if  the resource can be renewed,  the Agent harvests more to generate a higher profit since he can reduce the terminal tax by renewing  the resource at the end.
\appendix
\section{Proof of Proposition \ref{Approx-Hamiltonian}}
We first recall that $\Hc$ is defined by
\beqs
\mathcal H(y,\delta_1,\delta_2) & = & \sup_{z\in\R}\left\{ e^yp(e^y)a^*(e^y,z)-{a^*(e^y,z)^2\over 2}-\frac{ \sigma^2}{2}\gamma z^2 + (\lambda-{\sigma^2\over 2}-\mu(e^y)-a^*(e^y,z))\delta_1 \right\}\\
 & & 
+\frac{\sigma^2}{2}  \delta_2\;,\quad(y,\delta_1,\delta_2)\in \R\times\R\times\R\;.
\enqs
From the definition of $a^*$ given in \eqref{alphastar:thm} we can rewrite $ \Hc$ by considering  the different cases $e^y p(e^y) + z < - \underline M$, $-\underline M \leq e^y p(e^y) + z \leq \overline M$ and $e^y p(e^y) + z > \overline M$, and making the variable change $z - e^y p(e^y)$ for $z$  under the following form
\beqs
\mathcal H(y,\delta_1,\delta_2) & = & \max\Big\{\mathcal K_1(y,\delta_1) ,\mathcal K_2(y,\delta_1) ,\mathcal K_3(y,\delta_1) \Big\}+(\lambda-\frac{\sigma^2}2-\mu(e^y))\delta_1+\frac{\sigma^2}{2}  \delta_2\;,
\enqs
where
\beqs
\mathcal K_1(y,\delta_1) & = & \sup_{z\leq-\underline M}\left\{\left(\delta_1-e^y p (e^y)-{\underline M\over 2}\right)\underline M-{\sigma^2\over 2}\gamma\left(z-e^y p(e^y)\right)^2\right\}\\
 & = & \left(\delta_1-e^y p (e^y)-{\underline M\over 2}\right)\underline M-{\sigma^2\over 2}\gamma\left(\underline M+e^y p(e^y)\right)^2\;,
\enqs
\beqs
 \mathcal K_2(y,\delta_1) & = & \sup_{z\geq \overline M}\left\{\left(\delta_1-e^y p (e^y)-{\overline M\over 2}\right)\overline M-{\sigma^2\over 2}\gamma\left(z-e^y p(e^y)\right)^2\right\}\\
 & = & \left(-\delta_1+e^y p (e^y)-{\overline M\over 2}\right)\overline M-{\sigma^2\over 2}\gamma\left(\left[\overline M-e^y p(e^y)\right]_+\right)^2\;,
\enqs
and
$$
\mathcal K_3(y,\delta_1)  =  \sup_{-\underline M\leq z\leq \overline  M}\left\{-{1\over 2}(1+\gamma\sigma^2)z^2+\big(e^yp(e^y)(1+\gamma^2\sigma)-\delta_1\big)z\right\}-\frac{ \sigma^2}{2}\gamma |e^yp(e^y)|^2  
$$
for all $y,\delta_1\in\R$. 

A straightforward computation gives $\mathcal K_3(y,\delta_1) = Q(e^y p(e^y), \delta_1)$ where
\beqs
Q(\mathfrak{p},\delta_1) & = & 
{\big(\mathfrak{p}(1+\gamma^2\sigma)-\delta_1\big)^2\over 2(1+\gamma\sigma^2)}\mathds{1}_{\mathfrak{p}-{\delta_1\over 1+\gamma\sigma^2}\in[-\underline M,\overline M]}\\
 & & +\left(-{1\over 2}(1+\gamma\sigma^2)\overline M^2+\big(\mathfrak{p}(1+\gamma^2\sigma)-\delta_1\big)\overline M\right) \mathds{1}_{\mathfrak{p}-{\delta_1\over 1+\gamma\sigma^2}\in(\overline M,+\infty)} \\
& & 
+ \left(-{1\over 2}(1+\gamma\sigma^2)\underline M^2-\big(\mathfrak{p}(1+\gamma^2\sigma)-\delta_1\big)\underline M\right)\mathds{1}_{\mathfrak{p}-{\delta_1\over 1+\gamma\sigma^2}\in(-\infty,-\underline M)}\\
 & & -\frac{ \sigma^2}{2}\gamma |\mathfrak{p}|^2
\enqs
for all $\mathfrak p,\delta_1\in\R$.
\noindent

We then introduce the functions $\mathrm{abs}_\eps:~\R\rightarrow\R_+$ and 
 $\max_\eps:~\R^2\rightarrow\R$ defined by
\beqs
\mathrm{abs}_\eps(x) & = & |x|\Big(\Theta \Big(-\frac{4}{\varepsilon} x - 3 \Big) + \Theta \Big(\frac{4}{\varepsilon} x - 3 \Big) \Big)\;,\quad x\in\R\;,\\
\mathrm{max}_\eps (x,y) & = & {\mathrm{abs}_\eps (x-y)+x+y\over 2}\;,\quad x,y\in\R\;,
\enqs
for any $\eps>0$ where we recall that the function $\Theta$ is defined by (\ref{defTheta}). From the definition of $\Theta$,  the function $\mathrm{max}_\eps$ is infinitely differentiable with bounded derivatives and  we have
\beq\label{estim-maxeps}
\sup_{x,\,y\,\in\,\R}\big|\max(x,y)-\mathrm{max}_{\eps}(x,y)\big| & \leq & {\eps\over 3}
\enq
for all $\eps>0$.

Fix the constant $\Gamma  := \max\left\{ \overline M (1+\gamma\sigma^2), (P+\underline M) (1+\gamma\sigma^2) \right\}$. 
Since the function $Q$ is continuous on $\R^2$, there exists $Q_\eps$ infinitely differentiable on $\R^2$ such that
\beq\label{ineqQQeps}
\sup_{(\mathfrak{p},\delta_1)\in[0,P]\times [- (\Gamma+1), \; \Gamma+1]}|Q(\mathfrak{p},\delta_1)-Q_\eps(\mathfrak{p},\delta_1)| & \leq & {\eps\over 3}\;.
\enq
We then define $\Kc_{3,\eps}$ for any $y,\delta_1\in\R$ by
\beqs
\Kc_{3,\eps}(y,\delta_1) & = & Q_\eps(e^yp(e^y),\delta_1){\Big(1-\Theta(2(\delta_1-\Gamma) - 1)-\Theta\big(-2( \Gamma + \delta_1) -1)\big)\Big)}\\
 & & +\left(-{1\over 2}(1+\gamma\sigma^2)\underline M^2-\big(e^yp(e^y)(1+\gamma^2\sigma)-\delta_1\big)\underline M\right) {\Theta\big(2(\delta_1- \Gamma)-1\big)} \\
  & & +\left(-{1\over 2}(1+\gamma\sigma^2)\overline M^2+\big(e^yp(e^y)(1+\gamma^2\sigma)-\delta_1\big)\overline M\right) {\Theta\big(-2( \Gamma+\delta_1)-1\big)} \;.
\enqs
From \eqref{ineqQQeps} we have
$$
\sup_{\R^2} |\mathcal K_{3,\eps}-\mathcal K_3|  \leq  {\eps\over 3}\;.
$$

We then define for any $y,\delta_1,\delta_2\in \R$ the approximated Hamiltonian $\Hc_\eps$ by
$$
\mathcal H_\eps(y,\delta_1,\delta_2)  :=  \mathrm{max}_\eps\Big\{\mathcal K_1(y,\delta_1) ,\mathcal K_2(y,\delta_1) ,\mathcal K_{3,\eps}(y,\delta_1) \Big\}+\frac{\sigma^2}{2}  \delta_2 +(\lambda-\frac{\sigma^2}2-\mu(e^y))\delta_1 \;.
$$
We therefore get from \eqref{estim-maxeps}
$$
\sup_{\R^3} |\mathcal H_\eps-\mathcal H|  \leq  \eps\;.
$$
We then turn the PDE driven by $\Hc_\eps$ that writes
\begin{equation}\label{pde:w_eps}\begin{cases}
&-\partial_t w_\eps-\mathcal H_\eps\Big(y,\partial_{y} w_\eps(t,y), \partial_{yy} w_\eps(t,y)\Big)= 0\; ,\quad (t,y)\in [0,T)\times \mathbb R \;,\\
&w_\eps(T,y)= -f(e^y)\;,\quad y\in \mathbb R \;.
\end{cases}\end{equation}

\vspace{2mm}

\ni Under Assumption \textbf{(H')}, we can write the approximated Hamiltonian $\Hc_{3,\eps}$ under the form
$$
\Hc_{3,\eps} (y,\delta_1,\delta_2)  =   \frac{\sigma^2}{2}  \delta_2 + b_\eps(u,\delta_1)\delta_1+ c_\eps(y,\delta_1)
$$
where the coefficients $b_\eps$ and $c_\eps$ satisfy Conditions A, B and D of \cite{oleinik1961quasi}. Then, according to Theorem 14 in \cite{oleinik1961quasi}, PDE \eqref{pde:w_eps} admits a unique solution $w_\eps\in C^{2+\nu}([0,T]\times\R)$.

\end{document}